\documentclass[12pt]{amsart}

\usepackage[backend=biber, style=ext-alphabetic, sorting=nyt, articlein=false, backref=false, url=false, isbn=false, giveninits=true, maxbibnames=99]{biblatex}
\DeclareFieldFormat{postnote}{#1}
\DeclareFieldFormat{multipostnote}{#1}
\DeclareFieldFormat{eprint:hal}{%
  \mkbibacro{HAL}\addcolon\space
  \ifhyperref
    {\href{https://hal.science/#1}{\nolinkurl{#1}}}
    {#1}}

\usepackage{amssymb}
\usepackage[foot]{amsaddr}
\usepackage[margin=1in]{geometry}
\usepackage[hidelinks,pdfusetitle]{hyperref}
\usepackage{bm}
\usepackage[Symbol]{upgreek}
\usepackage{mathtools}
\usepackage{stmaryrd}
\usepackage{mathrsfs}
\usepackage{xcolor}
\usepackage{enumitem}
\setlist[enumerate, 1]{label=(\arabic*),leftmargin=4em}
\AtBeginEnvironment{thm}{\setlist[enumerate,1]{label=(\roman*),font=\upshape,leftmargin=4em}}
\AtBeginEnvironment{thmint}{\setlist[enumerate,1]{label=(\roman*),font=\upshape,leftmargin=4em}}
\AtBeginEnvironment{lem}{\setlist[enumerate,1]{label=(\roman*),font=\upshape,leftmargin=4em}}
\AtBeginEnvironment{prop}{\setlist[enumerate,1]{label=(\roman*),font=\upshape,leftmargin=4em}}
\AtBeginEnvironment{cor}{\setlist[enumerate,1]{label=(\roman*),font=\upshape,leftmargin=4em}}

\usepackage{lmodern}
\usepackage[T1]{fontenc}
\usepackage{microtype}

\DeclareMathOperator{\res}{res}

\DeclareMathOperator{\ddeg}{ddeg}
\DeclareMathOperator{\ndeg}{ndeg}

\DeclareMathOperator{\tp}{tp}
\DeclareMathOperator{\qftp}{qftp}

\DeclareMathOperator{\conv}{conv}
\DeclareMathOperator{\supp}{supp}

\def\d{\operatorname{d}}

\def\nl{\operatorname{nl}}
\def\dhl{\operatorname{dhl}}

\def\sm{\operatorname{small}}

\def\lift{\operatorname{lift}}
\def\pair{\operatorname{pair}}

\def\r{\operatorname{r}}
\def\sp{\operatorname{sp}}

\def\LE{\operatorname{LE}}

\newcommand{\On}{\mathsf{On}}
\newcommand{\No}{\mathsf{No}}

\newcommand{\OR}{\textnormal{OR}}

\newcommand{\flatter}{\mathrel{\prec\!\!\!\prec}}

\newcommand{\deft}[1]{\textbf{\textup{#1}}}
\newcommand{\ca}{\mathcal}

\newcommand{\llp}{(\!(}
\newcommand{\rrp}{)\!)}

\newcommand{\prece}{\preccurlyeq}
\newcommand{\succe}{\succcurlyeq}
\newcommand{\x}{\times}
\newcommand{\mf}{\mathfrak}

\newcommand{\0}{\emptyset}

\newcommand{\ges}{\geqslant}
\newcommand{\les}{\leqslant}
\newcommand{\N}{\mathbb{N}}

\newcommand{\Q}{\mathbb{Q}}
\newcommand{\R}{\mathbb{R}}
\newcommand{\T}{\mathbb{T}}
\newcommand{\HH}{\mathbb{H}}
\newcommand{\LL}{\mathbb{L}}

\newcommand{\dotrel}[1]{\mathrel{\dot{#1}}}

\DeclareFontFamily{U}{fsy}{}
\DeclareFontShape{U}{fsy}{m}{n}{<->s*[.9]psyr}{}
\DeclareSymbolFont{der@m}{U}{fsy}{m}{n}
\DeclareMathSymbol{\der}{\mathord}{der@m}{182}
\DeclareFontFamily{OMS}{smallo}{}
\DeclareFontShape{OMS}{smallo}{m}{n}{<->s*[.65]cmsy10}{}
\DeclareSymbolFont{smallo@m}{OMS}{smallo}{m}{n}
\DeclareMathSymbol{\cao}{\mathord}{smallo@m}{79}

\newtheorem{lem}{Lemma}[section]
\newtheorem{prop}[lem]{Proposition}
\newtheorem{cor}[lem]{Corollary}
\newtheorem{thm}[lem]{Theorem}

\theoremstyle{definition}

\newtheorem*{defn}{Definition}

\numberwithin{claim}{lem}
\numberwithin{equation}{section}

\usepackage{etoolbox}

\makeatletter
\patchcmd{\@startsection}
  {\@afterindenttrue}
  {\@afterindentfalse}
  {}{}
\makeatother

\title{Tame pairs of transseries fields}
\author{Nigel Pynn-Coates}
\address{Kurt G\"{o}del Research Center for Mathematical Logic, University of Vienna, Kolingasse 14-16, 1090 Wien, Austria}
\email{\href{mailto:nigel.pynn-coates@univie.ac.at}{nigel.pynn-coates@univie.ac.at}}

\begin{document}

\begin{abstract}
This paper concerns pairs of models of the theory of the differential field of logarithmic-exponential transseries that are tame as a pair of real closed fields. That is, the smaller model is bounded inside the larger model and there exists a standard part map. This covers for instance the differential fields of hyperseries or surreal numbers or maximal Hardy fields equipped with suitable enlargements of the differential field of transseries.
The theory of such pairs is complete and model complete in a natural language and it has quantifier elimination in the same language expanded by two predicates and a standard part map; it is also distal and hence has NIP.
Additionally, the smaller model is purely stably embedded in the pair, and hence so is the constant field.

More generally, we study differential-Hensel-Liouville closed pre-$H$-fields, i.e., pre-$H$-fields that are differential-henselian, real closed, and closed under exponential integration, equipped with lifts of their differential residue fields, and establish similar results in that setting relative to the differential residue field.
\end{abstract}
\maketitle


\section{Introduction}

Transseries were introduced independently by Dahn and G\"{o}ring in their work on nonstandard models of the real exponential field \cite{dahngoering} and by \'{E}calle in his solution of Dulac's Conjecture, connected to Hilbert's 16th problem \cite{ecalle1,ecalle2}.
Building on these constructions, the differential field of logarithmic-exponential transseries, denoted here by $\T$, was constructed in \cite{dmm97} (there denoted by $\R\llp t \rrp^{\LE}$).
It is a field of formal series with an intricate construction (see also \cite[Appendix~A]{adamtt}), and it has various closure properties.
For example, $\T$ is a real closed field that is closed under exponentiation, integration, and compositional inversion.
As a differential field, $\T$ also defines an ordering and a valuation ring, where the latter is the convex hull of $\R \subseteq \T$, and we consider $\T$ as an ordered valued differential field.

Studying $\T$ from the perspective of model theory, \cite{adamtt} establishes that $\T$ is \emph{model complete} as an ordered valued differential field.
Model completeness can be viewed as an abstract analogue of Hilbert's Nullstellensatz, and indeed the familiar Nullstellensatz can be deduced from the model completeness of algebraically closed fields; likewise, real closed fields are model complete, yielding a real Nullstellensatz.
Model completeness for $\T$ as an ordered valued differential field is thus informally a kind of Nullstellensatz for systems of algebraic differential equations with asymptotic conditions on solutions that can be solved in an ordered setting.
In other words, this result together with others from \cite{adamtt}, which won the 2018 Karp Prize from the Association for Symbolic Logic, substantiate the conjecture from \cite{dmm01}, made precise in \cite{adh-toward}, that $\T$ forms a universal domain for ordered differential algebra.
Among the other results is a complete and effective axiomatization of the elementary properties of~$\T$. 

Since those results, there has been interest in larger models of the theory of $\T$, such as hyperseries, surreal numbers, and maximal Hardy fields.
In contrast to $\T$, in which every element is bounded in absolute value by some finite compositional iterate $\exp_n(x)$ of $\exp(x)$, each of the above contains elements that are transexponential; this is clear for $\HH$ and $\No$, but for maximal Hardy fields is a result of Boshernitzan \cite{bosher-neworders,bosher-transexpHf}.
In this sense, $\T$ is small, which may seem perverse since $\T$ is itself a nonstandard real closed field (or even exponential field) extension of $\R$.
But taking the theory of $\T$ as a differential field as our starting point, $\T$ becomes the standard model. 
As ordered valued differential fields, $\HH$, $\No$, and maximal Hardy fields have exactly the same elementary properties as $\T$, and thus they are model complete.
On the other hand, their transexponential elements cannot be detected by these elementary properties.

The differential field $\HH \supseteq \T$ of hyperseries was constructed as a field in \cite{hyperseries}, building on \cite{loghyperseries}, and as an elementary differential field extension of $\T$ in \cite{hyperseries-deriv}.
Like $\T$, it is a field of formal series, but just to highlight a couple of novel points, $\HH$ is a proper class that contains formally transexponential elements such as $E_{\omega}(x)$, which in $\HH$ satisfies the functional equation $E_{\omega}(x+1)=\exp(E_{\omega}(x))$ that cannot be solved in $\T$.
Similarly, Conway's field $\No$ of surreal numbers from \cite{conway-surreals} (see also \cite{gonshor-surreals}) is a proper class containing every ordinal number.
A suitable derivation on $\No$ was constructed in \cite{bm-surreals}, and \cite{adh-surreals} shows that there is an elementary embedding $\T \to \No$ of differential fields sending $x$ to the ordinal~$\omega$.

More concretely, recall that a \emph{Hardy field} is a field of germs at $+\infty$ of eventually differentiable functions $\R \to \R$ that is closed under differentiation, and a \emph{maximal} Hardy field is maximal with respect to inclusion.
Then $\T$ can also be embedded into any maximal Hardy field 
\cite[Corollary~8.10]{ad-analHf}, and this embedding is elementary by the difficult theorem \cite[Theorem~11.19]{adh-maxHftheory} (which also uses the main results of \cite{adh-normalizationADA,adh-omegafreeHf}).
One consequence is that a system of algebraic differential equations over $\Q$ has a solution in some Hardy field if and only if it has a solution in $\T$; in that case, it will have a solution in every maximal Hardy field.
See \cite{adh-maxHftheory} for stronger versions of these results and others.

In order to study these large elementary extensions of $\T$, we can enrich $\HH$ by the convex hull $\dot{\ca O}_{\HH}$ of $\T$ in $\HH$, which consists exactly of those elements of $\HH$ bounded in absolute value by some $\exp_n(x)$. 
This $\dot{\ca O}_{\HH}$ is a second valuation ring of $\HH$ properly containing the natural valuation ring $\ca O_{\HH}$ of $\HH$, namely the convex hull of $\R$ in $\HH$.
It therefore allows us to express in elementary terms the property of being exponentially bounded in $\HH$, in the same way that $\ca O_{\HH}$ expresses the property of being bounded by a real number.
These yield asymptotic relations comparing sizes of elements with respect to different scales.
In $\No$ one can take $\dot{\ca O}_{\No}$ to be the convex hull of the image of the elementary embedding $\T \to \No$ and $\ca O_{\No}$ to be the convex hull of $\R$ in $\No$, and for $H$ a maximal Hardy field define $\dot{\ca O}_{H}$ and $\ca O_{H}$ likewise.
We thus strengthen the model completeness of $\HH$, $\No$, and maximal Hardy fields as follows.
\begin{thm}\label{thmint:HNomodcomp}
The structures $(\HH, \ca O_{\HH}, \dot{\ca O}_{\HH})$, $(\No, \ca O_{\No}, \dot{\ca O}_{\No})$, and $(H, \ca O_{H}, \dot{\ca O}_{H})$ for $H$ any maximal Hardy field are model complete.
\end{thm}

This result is a consequence of our main theorem about transserial tame pairs, defined shortly.
Consider the pair $(\HH, \T)$ of differential fields.
Not every element $f \in \dot{\ca O}_{\HH}$ has a \emph{standard part} in $\T$, which would be a (unique) element $g\in \T$ to which $f$ is infinitely close (in the sense that $f-g$ is in the maximal ideal of $\dot{\ca O}_{\HH}$).
For instance, consider the compositional inverse $L_{\omega}(x) \in \HH$ of $E_{\omega}(x)$.
Thus there is no \emph{standard part map} $\dot{\ca O}_{\HH} \to \T$ sending $f \in \dot{\ca O}_{\HH}$ to its standard part in $\T$, or in other words, $(\HH, \T)$ is not \emph{tame} as a pair of real closed fields, the study of which goes back to \cite{macintyre-thesis} (see also \cite[Section~3.6]{adamtt} for an exposition of results about tame pairs of real closed fields and \cite{dl-Tconvex} for tame pairs of o-minimal structures).
But as we show in Corollary~\ref{cor:largeTnlexTTP}, $\T$ can be elementarily extended to $\T^* \subseteq \dot{\ca O}_{\HH}$ so that $(\HH, \T^*)$ is tame.
The same is true in $\No$ and maximal Hardy fields.
In Section~\ref{sec:hyperseries} we give a more explicit example of a $\T^* \subseteq \dot{\ca O}_{\HH}$ so that $(\HH, \T^*)$ is tame.

A \deft{transserial tame pair} is a pair $(K, L)$ of differential fields such that $K$ and $L$ are models of the theory of $\T$ as differential fields, $L$ is a proper differential subfield of $K$, and $(K, L)$ is tame as a pair of real closed fields.
Thus the pair $(\HH, \T^*)$ above is a transserial tame pair, as is any tame pair of Hardy fields that satisfy the theory of $\T$.
Our goal is to study the model theory of such pairs.
It turns out (see Lemma~\ref{lem:elemext}) that construing $K$ and $L$ as ordered valued differential fields with their natural valuation rings, i.e., the convex hulls of the constant field, $K$ is an elementary extension of $L$.
In the introduction we thus construe a transserial tame pair $(K, L)$ as a pair of ordered valued differential fields (in the body we are more explicit).
Then:

\begin{thm}[Corollary~\ref{cor:modcompTTP}]
Every transserial tame pair $(K, L)$ expanded by a predicate for the convex hull of $L$ in $K$ is model complete.
\end{thm}

Let $(K, L)$ be a transserial tame pair and $\dot{\ca O}$ be the convex hull of $L$ in $K$.
The proof of the above theorem makes use both of the model completeness of the theory of $\T$ as an ordered valued differential field and of the model theory and algebra of valuation-theoretic coarsenings $(K, \dot{\ca O})$ of $K$.
These latter structures $(K, \dot{\ca O})$ are differential-Hensel-Liouville closed pre-$H$-fields (see Section~\ref{sec:prelim} for definitions and Lemma~\ref{lem:TTPcoarsenTdhl}), and their model theory and algebra was developed in \cite{pc-preH-gap}, which relies crucially on \cite{pc-dh} and of course considerable parts of the machinery of \cite{adamtt}.

We can obtain the theorem about $\HH$ and $\No$ above from the result about transserial tame pairs by considering structures of the form $(K, \ca O, \dot{\ca O})$, where $\ca O$ is the natural valuation ring of $K$ and $\dot{\ca O}$ is the convex hull of $L$ in $K$.
Model completeness for these structures is Theorem~\ref{thm:modcompTnldhl}.
We expect that $L$ is not definable in $(K, \ca O, \dot{\ca O})$ (see \cite{pc-preH-dim} for some related results), in which case $(K, \ca O, \dot{\ca O})$ is a proper reduct of $(K, L)$ in the sense of definability.
In fact, the theory of such structures $(K, \ca O, \dot{\ca O})$ can be axiomatized in the language of ordered differential fields together with predicates for $\ca O$ and $\dot{\ca O}$, without reference to $L$.
In Section~\ref{sec:coarsenuncoarsen} we show how to pass between a transserial tame pair $(K, L)$, such a structure $(K, \ca O, \dot{\ca O})$, and a structure $(K, \dot{\ca O})$ together with its ordered differential residue field, which are thus three ways of considering essentially the same objects.

In fact, the model completeness of transserial tame pairs is a consequence of more general results about differential-Hensel-Liouville closed pre-$H$-fields equipped with lifts of their differential residue fields.
These differential residue fields may be expanded by additional structure, such as a valuation ring (which is necessary to obtain model completeness for transserial tame pairs).
Section~\ref{sec:modcomp} establishes model completeness relative to the possibly enriched differential residue field, as well as the model completeness of the structures $(K, \ca O, \dot{\ca O})$ mentioned above.

In Section~\ref{sec:AKErelQE} we refine our techniques and results, obtaining relative completeness, relative model completeness, and relative quantifier elimination with a standard part map, where ``relative'' is in the sense of the previous paragraph.
These results are analogous to two-sorted results from \cite{pc-preH-gap}.
Such relative completeness is often called an Ax--Kochen/Ershov theorem, after their results about henselian valued fields \cite{ak1,ak3,ershov}, and has here the following consequence for transserial tame pairs.
\begin{thm}[Corollary~\ref{cor:TTPcomp}]
The theory of transserial tame pairs is complete.
\end{thm}
The relative quantifier elimination similarly yields a quantifier elimination for transserial tame pairs by expanding the language by two predicates needed for quantifier elimination for $\T$ (see \cite[Chapter~16]{adamtt}) and by a standard part map.
For precise statements, see Section~\ref{sec:AKErelQE}.
Finally, we mention two more results.
The first is a pure stable embeddedness result for transserial tame pairs, in the sense that no new structure is induced on the smaller model or the constant field:
\begin{thm}[Corollary~\ref{cor:Cstabembed}]
Let $(K, L)$ be a transserial tame pair and $C$ be the constant field of $K$.
Then any subset of $L^n$ that is definable in $(K, L)$ with parameters from $K$ is definable in $L$ with parameters from $L$, and any subset of $C^n$ that is definable in $(K, L)$ with parameters from $K$ is definable in $C$ with parameters from~$C$.
\end{thm}
The next result shows that transserial tame pairs are tame in the sense of model-theoretic dividing lines.
Roughly, the class of distal structures consists of those structures with NIP that are ``purely unstable''. 
\begin{thm}[Theorem~\ref{thm:TTPdistal}]
The theory of transserial tame pairs is distal and so has NIP.
\end{thm}

\section{Preliminaries and notation}\label{sec:prelim}

We let $n$ range over $\N = \{0, 1, 2, \dots\}$.

\subsection{Valued differential fields and differential-henselianity}

Let $K$ be a differential field, equipped with the derivation $\der \colon K \to K$.
Attached to $K$ is an important subfield, the \deft{constant field} of $K$, denoted by $C \coloneqq \{ f \in K : \der(f)=0 \}$; to indicate the dependence on $K$, we write $C_K$.
For $f \in K$, we often write $f'$ for $\der(f)$ if the derivation is clear from context and set $f^\dagger \coloneqq f'/f$ if $f \neq 0$, the logarithmic derivative of $f$.
We say that $K$ is \deft{closed under integration} if $\der K = K$
and \deft{closed under exponential integration} if $(K^{\x})^\dagger = K$.
We let $K\{Y\} \coloneqq K[Y, Y', Y'', \dots]$ be the ring of differential polynomials over $K$ and set $K\{Y\}^{\neq} \coloneqq K\{Y\} \setminus \{0\}$.
For $P \in K\{Y\}^{\neq}$, the \deft{order} of $P$ is the smallest $m \in \N$ such that $P \in K[Y, Y', \dots, Y^{(m)}]$ and $P_n$ is the homogeneous part of $P$ of degree $n$, which we will only need for $n=0$ and $n=1$; degree for differential polynomials means total degree.
We call $K$ \deft{linearly surjective} if for all $a_0, \dots, a_n \in K$ with $a_n \neq 0$, the equation $1 + a_0y + a_1y' + \dots + a_n y^{(n)} = 0$ has a solution in $K$.
If $L$ is a differential field extension of $K$ and $a \in L$, then $K \langle a \rangle$ denotes the differential subfield of $L$ generated by $a$ over~$K$.

Now let $(K, \ca O)$ be a \deft{valued differential field} in the sense of \cite[Section~4.4]{adamtt}, which simply means that $K$ is a differential field and $\ca O \supseteq \Q$ is a valuation ring of $K$, i.e., $\ca O$ is a subring of $K$ that contains $a$ or $a^{-1}$ for every $a \in K^{\x}$.
In this paper, there will often be two valuation rings on a field equipped with a single derivation, so in the notation, we specify $\ca O$ but leave $\der$ implicit.
With $\ca O^{\x} = \{ a \in K^{\x} : a, a^{-1} \in \ca O \}$, the ring $\ca O$ has a unique maximal ideal $\cao \coloneqq \ca O \setminus \ca O^{\x}$.
We introduce the following binary relations, for $f, g \in K$:
\begin{align*}
f \prece g\ &\Leftrightarrow\ f \in \ca Og, & f \asymp g\ &\Leftrightarrow\ f \prece g\ \text{and}\ g \prece f,\\
f \prec g\ &\Leftrightarrow\ f \in \cao g\ \text{and}\ g\neq 0, & f\sim g\ &\Leftrightarrow\ f-g \prec g.
\end{align*}
The relation $\asymp$ is an equivalence relation on $K$, the relation $\sim$ is an equivalence relation on $K^{\x}$, and they satisfy: if $f \sim g$, then $f \asymp g$.
(For relations like $\prece$ and their connection to valuations, see for example \cite[Definition~3.1.1]{adamtt} and subsequent remarks.)
When there is a second valuation ring of $K$ denoted by $\dot{\ca O}$, we distinguish the above as $\dot{\cao}$, $\dotrel{\prece}$, etc.
The \deft{residue field} of $(K, \ca O)$ is $\res(K, \ca O) \coloneqq \ca O/\cao$.
By our assumption that $\Q \subseteq \ca O$, the characteristics of $K$ and $\res(K, \ca O)$ are~$0$.

To define differential-henselianity below and in two lemmas in Section~\ref{sec:coarsenuncoarsen}, we extend the relations displayed above to $K\{Y\}$, for which it is convenient to work with valuations instead of valuation rings.
It is well-known that to $(K, \ca O)$ is associated a surjective \emph{valuation} $v \colon K^{\x} \to \Gamma$, where $\Gamma$ is an ordered abelian group called the \deft{value group} of $K$ (and conversely, from such a valuation one gets a valuation ring of $K$).
Adding a new symbol $\infty$ to the value group $\Gamma$ and extending the addition and ordering to $\Gamma_\infty \coloneqq \Gamma \cup \{\infty\}$ by $\infty+\gamma=\gamma+\infty=\infty$ and $\infty>\gamma$ for all $\gamma \in \Gamma$ allows us to extend $v$ to $K$ by setting $v(0) \coloneqq \infty$.
Then for all $f, g \in K$,
\[
f \prece g\ \Leftrightarrow\ vf \ges vg\ \qquad\ \text{and}\ \qquad\ f \prec g\ \Leftrightarrow\ vf > vg.
\]
Thus setting $v(P)$ to be the minimum valuation of the coefficients of $P \in K\{Y\}$, we can extend the relations $\prece$, $\prec$, $\asymp$, and $\sim$ to $K\{Y\}$.

One basic condition relating the valuation and the derivation is called \deft{small derivation}, which means for $(K, \ca O)$ that $\der\cao \subseteq \cao$.
In this case, $\der\ca O \subseteq \ca O$ \cite[Lemma~4.4.2]{adamtt}, so $\der$ induces a derivation on $\res(K, \ca O)$, and we always construe $\res(K, \ca O)$ as a differential field with this induced derivation.
(Small derivation is a strong form of continuity of the derivation with respect to the valuation topology, for which see \cite[Lemma~4.4.7]{adamtt}.)
Note that if $(K,\ca O)$ has small derivation and $\ca O = C + \cao$, then the derivation induced on $\res(K,\ca O)\cong C$ is trivial.
In this paper, $(K,\ca O)$ will usually satisfy $\ca O = C + \cao$ (see the next subsection), while $(K,\dot{\ca O})$ will usually have small derivation and nontrivial derivation on $\res(K,\dot{\ca O})$.

An analogue of henselianity of a valued field for a valued differential field with small derivation is differential-henselianity, which we define now.
The distinct notion of newtonianity (see \cite[Chapter~14]{adamtt}), which is relevant for valued differential fields like $\T$, will be introduced in Section~\ref{subsec:newtonomega}.
In light of the above, $(K,\dot{\ca O})$ is a valued differential field for the rest of this subsection.

\begin{defn}
We call $(K, \dot{\ca O})$ \deft{differential-henselian} (\deft{$\d$-henselian} for short) if $(K, \dot{\ca O})$ has small derivation and:
\begin{enumerate}[label=(DH\arabic*)]
	\item the differential field $\res(K, \dot{\ca O})$ is linearly surjective;
	\item if $P \in K\{Y\}$ satisfies $P_0 \dotrel{\prec} 1$ and $P \dotrel{\asymp} P_1 \dotrel{\asymp} 1$, there is $y \dotrel{\prec} 1$ in $K$ with $P(y) = 0$.
\end{enumerate}
\end{defn}
Note that $P_0=P(0)$ and $P_1 \dotrel{\asymp} 1$ means $\frac{\partial P}{\partial Y^{(n)}}(0) \dotrel{\asymp} 1$ for some $n$ at most the order of $P$.
Differential-henselianity was introduced in \cite{scanlon} and developed systematically in \cite[Chapter~7]{adamtt}.
We call $K$ \deft{Liouville closed} if $K$ is real closed, closed under integration, and closed under exponential integration.
We call $(K, \dot{\ca O})$ \deft{differential-Hensel-Liouville closed} (shorter: \deft{$\d$-Hensel-Liouville closed}) if $(K, \dot{\ca O})$ is $\d$-henselian and $K$ is Liouville closed.
(If $(K, \dot{\ca O})$ is $\d$-henselian, then closure under integration comes for free by \cite[Lemma~7.1.8]{adamtt}.)

A \deft{lift} of $\res(K, \ca O)$ is a subfield $\bm k \subseteq \ca O$, in the sense that $\bm k$ is a subring of $\ca O$ that is itself a field, that maps isomorphically as a field onto $\res(K, \ca O)$ under the residue map; equivalently, for every $a \in \ca O^{\x}$ there is a (necessarily unique) $u \in \bm k^{\x}$ with $a \sim u$ (i.e., $a-u \in \cao$).
Recall that if $(K, \ca O)$ is henselian as a valued field, then $(K, \ca O)$ always admits a lift of its residue field \cite{kaplansky1,maclane-powerseries} (or see \cite[Proposition~3.3.8]{adamtt}).
In case $(K, \dot{\ca O})$ is $\d$-henselian, it can be equipped with a lift of its \emph{differential} residue field \cite[Proposition~7.1.3]{adamtt}, which means that there exists $\bm k \subseteq \dot{\ca O}$ that is a differential subfield of $\dot{\ca O}$ and maps isomorphically as a differential field onto $\res(K, \dot{\ca O})$ under the residue map.
By the proof of that proposition:
\begin{prop}[{\cite[Proposition~7.1.3]{adamtt}}]\label{adh:7.1.3}
If $(K, \dot{\ca O})$ is $\d$-henselian, then any differential subfield of $\dot{\ca O}$ can be extended to a lift of the differential residue field $\res(K,\dot{\ca O})$. 
\end{prop}
These lifts play an important role in this paper.

\subsection{Ordered differential fields and (pre-)\texorpdfstring{$H$}{H}-fields}\label{subsec:preH}

Let $K$ be an ordered differential field, in the sense that $K$ is a differential field additionally equipped with an ordering $\les$ making it an ordered field (i.e., the ordering is compatible with addition and multiplication in the usual way).
For $A \subseteq K$, we let $\conv_K(A) \coloneqq \{ b \in K : a_1 \les b \les a_2\ \text{for some}\ a_1, a_2 \in A \}$ denote the convex hull of $A$ in $K$.
Then $K$ can always be equipped with its so-called natural valuation ring $\conv_K(C)$, which henceforth we denote by $\ca O$, while another valuation ring of $K$ will be denoted by $\dot{\ca O}$.
For another ordered differential field $L$ we write $\ca O_L$ and~$\dot{\ca O}_L$.

The valuation ring $\ca O$ is always existentially definable in the ordered differential field $K$ without parameters.
Certain model-theoretic statements are therefore independent of the distinction between $K$ and $(K, \ca O)$: For example, $K \prece L$ if and only if $(K, \ca O) \prece (L, \ca O_L)$.
On the other hand, it may be that $L$ is a differential field extension of $K$ but $(L, \ca O_L)$ is not a valued differential field extension of $(K, \ca O)$, so caution will be taken to specify valuation rings when necessary.

Typically, we expect nice interaction between the ordering and the derivation.
For instance, we call $K$ an \deft{$H$-field} if:
\begin{enumerate}[label=(H\arabic*)]
    \item for all $a \in K$, if $a>C$, then $a'>0$;
    \item $\ca O = C + \cao$.
\end{enumerate}
The second condition says that $C$ is a lift of $\res(K, \ca O)$ as a field, and also yields $\der\ca O\subseteq \cao$ if $(K,\ca O)$ has small derivation.
For example, the differential field $\T$ is an $H$-field, and indeed $H$-fields were introduced in \cite{ad-hf} towards axiomatizing the theory of $\T$.
As described above, when convenient we construe an $H$-field $K$ as an ordered valued differential field $(K, \ca O)$.
A related notion is the following.
We call $(K, \dot{\ca O})$ a \deft{pre-$H$-field} if:
\begin{enumerate}[label=(PH\arabic*)]
    \item\label{ph1} $\dot{\ca O}$ is convex (with respect to $\les$);
    \item\label{ph2} for all $f \in K$, if $f > \dot{\ca O}$, then $f'>0$;
    \item\label{ph3} for all $f, g \in K^{\x}$ with $f \dotrel{\prece} 1$ and $g \dotrel{\prec} 1$, we have $f' \dotrel{\prec} g^\dagger$.
\end{enumerate}
In this definition, $\dot{\ca O}$ may be any valuation ring of $K$ with associated relations $\dotrel{\prece}$ and $\dotrel{\prec}$, but \ref{ph2} forces $C \subseteq \dot{\ca O}$, so $\ca O \subseteq \dot{\ca O}$ by \ref{ph1}.
Recall that \ref{ph1} holds if and only if $\dot{\cao}$ is convex, so if \ref{ph1} holds, then $\les$ induces an ordering on $\res(K, \dot{\ca O})$ making it an ordered field.
We thus construe the residue field of a pre-$H$-field with small derivation as an ordered differential field.
Construing an $H$-field $K$ as an ordered valued differential field $(K, \ca O)$, it is a pre-$H$-field, where \ref{ph3} holds by \cite[Lemma~10.5.1]{adamtt}.
Pre-$H$-fields are so named because every ordered valued differential subfield of an $H$-field is a pre-$H$-field and moreover every pre-$H$-field can be extended to an $H$-field by \cite[Corollary~4.6]{ad-hf} (also see \cite[Corollary~10.5.13]{adamtt}).

We have already mentioned that $\T$ is an $H$-field, but it is also a Liouville closed $H$-field.
Although it has small derivation, it is not $\d$-henselian because its residue field is isomorphic to $\R$ with the trivial derivation.
Instead, it is \emph{newtonian}, a more subtle and more technical differential analogue of henselianity defined next.

\subsection{Newtonianity and \texorpdfstring{$\upomega$}{ω}-freeness}\label{subsec:newtonomega}

We define what it means for an $H$-field to be newtonian and $\upomega$-free, critical properties of $\T$ that appear directly only in Lemmas~\ref{lem:uncoarsen-newtonian} and \ref{lem:uncoarsen-omegafree}.
Ignoring those two proofs will allow the reader to skip this subsection.
For more on newtonianity, see \cite[Chapter~14]{adamtt}; for more on $\upomega$-freeness, see \cite[Sections~11.7, 11.8, 13.6]{adamtt}.

Before defining newtonianity, we need to explain some additional induced structure on the value group of a pre-$H$-field.
Let $(K, \ca O)$ be a pre-$H$-field with $\ca O \neq K$.
Then logarithmic differentiation in $K$ induces a map defined by, for $g \in K^{\x}$ with $vg\neq 0$,
\begin{align*}
\psi \colon \Gamma^{\neq} &\to \Gamma\\
vg &\mapsto v(g^\dagger)
\end{align*}
(this map makes sense by \cite[Proposition~9.1.3 and Lemma~10.1.1]{adamtt}).
We call $(\Gamma, \psi)$ the \deft{asymptotic couple} of $K$; such structures were introduced by Rosenlicht \cite{rosen-dval}, and more about them can be found in \cite[Sections~6.5, 9.2]{adamtt}.
Asymptotic couples appear a little in the rest of this section and more in Lemmas~\ref{lem:uncoarsen-newtonian} and \ref{lem:uncoarsen-omegafree}.

It is worth noting that asymptotic couples make sense in a more general context, namely asymptotic fields, for which see \cite[Chapter~9]{adamtt}. Likewise, the rest of the subsection, including newtonianity and $\upomega$-freeness, make sense in the context of $H$-asymptotic fields. In particular, the ordering on $K$ is not used in this subsection.
Rather, we have assumed that $(K, \ca O)$ is a pre-$H$-field to avoid introducing additional definitions that are not important to this paper.
In appealing to results from \cite{adamtt} about asymptotic couples, note that every asymptotic couple $(\Gamma, \psi)$ in this paper is of $H$-type in the sense that $0<\alpha<\beta$ in $\Gamma$ implies $\psi(\alpha)\ges\psi(\beta)$ (also called $H$-asymptotic couples).

To define newtonianity, we recall the notion of newton degree from \cite[Sections~11.1 and 11.2]{adamtt}.
Now we suppose that $\psi(\Gamma^{\neq})$ has no maximum.
We call $\phi \in K^{\x}$ \deft{active} (tacitly, in $(K, \ca O)$) if $v\phi \les \psi(\gamma)$ for some $\gamma \in \Gamma^{\neq}$.
Below, we let $\phi$ range over active elements of $K$.
Then the differential field $K^{\phi}$ is the field $K$ whose derivation $\der$ is replaced by its multiple $\phi^{-1}\der$, so the constant field of $K^{\phi}$ remains $C$.
If $\phi>0$, then $(K^{\phi}, \ca O)$ remains a pre-$H$-field; if $\phi>0$ and $K$ is moreover an $H$-field, then so is $K^{\phi}$.
The asymptotic couple of $(K^{\phi}, \ca O)$ is $(\Gamma, \psi-v\phi)$, where $(\psi-v\phi)(\gamma)=\psi(\gamma)-v\phi$ for $\gamma \in \Gamma^{\neq}$, so $(\psi-v\phi)(\Gamma^{\neq})$ still has no maximum.
What we have arranged is that $(K^{\phi}, \ca O)$ has small derivation by \ref{ph3} and $\{ \gamma \in \Gamma^{\neq} : (\psi-v\phi)(\gamma)>0 \} \neq \0$.

This procedure leads to the ring $K^{\phi}\{Y\}$ of differential polynomials over $K^{\phi}$, with evaluation taking place in the differential field $K^{\phi}$.
We have a ring isomorphism $K\{Y\} \to K^{\phi}\{Y\}$ given by associating to $P \in K\{Y\}$ an appropriate element $P^{\phi} \in K^{\phi}\{Y\}$, with the property that $P^{\phi}(y) = P(y)$ for all $y \in K$.
The details of this map are not important here and can be found in \cite[Section~5.7]{adamtt}, but it is the identity on the common subring $K[Y] = K^{\phi}[Y]$ of $K\{Y\}$ and $K^{\phi}\{Y\}$.
Consider $P^{\phi} \in K^{\phi}\{Y\}^{\neq}$ and take $a \in K^{\x}$ with $va=-vP^{\phi}$.
Then we have the differential polynomial $\overline{aP^{\phi}} \in \res(K^{\phi}, \ca O)\{Y\}^{\neq}$ obtained by applying the residue map to $aP^{\phi}$, which makes sense since $(K^{\phi}, \ca O)$ has small derivation (although $\res(K^{\phi}, \ca O)$ is just the field $\res(K, \ca O)$ equipped with the trivial derivation), and we let $\ddeg P^{\phi}$ be the (total) degree of $\overline{aP^{\phi}}$.
What is important here is that $\ddeg P^{\phi}$ eventually stabilizes, meaning that there is an active $\phi_0 \in K^{\x}$ such that for all (active) $\phi \prece \phi_0$, $\ddeg P^\phi=\ddeg P^{\phi_0}$.
We denote this eventual value of $\ddeg P^\phi$ by $\ndeg P$.
With this, we can define newtonianity, one of the most consequential elementary properties of $\T$, which includes the assumption above that $\psi(\Gamma^{\neq})$ has no maximum.
\begin{defn}
We call $(K, \ca O)$ \deft{newtonian} if each $P \in K\{Y\}$ with $\ndeg P = 1$ has a zero in~$\ca O$.
\end{defn}

Another important property of $\T$ is that it is $\upomega$-free.
For this definition, we continue to assume that $(K, \ca O)$ is a pre-$H$-field with $\ca O \neq K$ such that $\psi(\Gamma^{\neq})$ has no maximum.
\begin{defn}
We say that $(K, \ca O)$ is \deft{$\upomega$-free} if for every $f \in K$, there exists $g \succ 1$ in $K$ with $f - 2(g^{\dagger\dagger})' + (g^{\dagger\dagger})^2 \succe (g^{\dagger})^2$.
\end{defn}

\subsection{The theory of \texorpdfstring{$\T$}{T}}
Let $T^{\nl}_{\sm}$ be the theory of $\upomega$-free, newtonian, Liouville closed $H$-fields with small derivation.
This theory completely axiomatizes the theory of the differential field $\T$; the most difficult of these properties to establish for $\T$ is newtonianity, for which see \cite[Corollary~15.0.2]{adamtt}.
As explained in Section~\ref{subsec:preH}, the natural valuation ring of an $H$-field is definable in its differential field structure without parameters and the ordering of a real closed $H$-field is definable in its field structure without parameters, so the theory $T^{\nl}_{\sm}$ can be formulated more naturally in the language $\{ +, -, \cdot, 0, 1, \der, \les, \ca O \}$ of ordered valued differential fields or alternatively in the language $\{ +, -, \cdot, 0, 1, \der \}$ of differential fields.

One of the main results of \cite{adamtt} is that $T^{\nl}_{\sm}$ is model complete in the language $\{ +, -, \cdot, 0, 1, \der, \les, \ca O \}$ by \cite[Corollary~16.2.5]{adamtt} (in fact, the theory $T^{\nl}$ without the assumption ``small derivation'' is already model complete, and $T^{\nl}_{\sm}$ is one of its two completions \cite[Corollary~16.6.3]{adamtt}).
To be precise, in \cite[Corollary~16.2.5]{adamtt} model completeness is stated with a binary relation $\prece$ replacing the unary predicate $\ca O$, but it is easy to see that $\ca O$ is enough for model completeness ($\prece$ is important for quantifier elimination if multiplicative inversion is not in the language).
But $T^{\nl}_{\sm}$ is not model complete without the valuation ring, because the valuation ring of $\T$ is not universally definable (allowing parameters) in the differential field $\T$ \cite[Corollary~16.2.6]{adamtt}.
Hence, to obtain model completeness for transserial tame pairs, it is critical to put the valuation ring in the language.

\section{Three perspectives on transserial tame pairs}\label{sec:coarsenuncoarsen}

\subsection{Introduction}
\begin{defn}
A pair $(K, L)$ of differential fields is a \deft{transserial tame pair} if:
\begin{enumerate}[label=(TTP\arabic*)]
    \item $K, L \models T^{\nl}_{\sm}$ as differential fields;
    \item $L$ is a proper differential subfield of $K$;
    \item $L$ is tame in $K$ as real closed fields, i.e., $\dot{\ca O} = L + \dot{\cao}$, where $\dot{\ca O}=\conv_{K}(L)$ and $\dot{\cao}$ is its maximal ideal.
\end{enumerate}    
\end{defn}
Although a transserial tame pair $(K, L)$ was defined above as a pair of differential fields, it could have been defined similarly as a pair of ordered valued differential fields.
Regardless, it turns out that $K$ must be an elementary extension of $L$ as differential (equivalently, ordered valued differential) fields, as we now show.
We use this fact often in what follows.

\begin{lem}\label{lem:elemext}
Let $(K, L)$ be a transserial tame pair.
Then $C = C_L$ and so $(K, \ca O)$ is an elementary extension of $(L, \ca O_L)$.
\end{lem}
\begin{proof}
First, note that $C \subseteq \ca O \subseteq \dot{\ca O}$.
To observe the latter inclusion, take any $a \in L$ with $a' \in C_L^{\x}$.
Then since $K$ is an $H$-field, we have $a \succ 1$ in $K$, so $|a|>\ca O$ and thus $\ca O \subseteq \dot{\ca O}$.
Let $c \in C$, so $c=a-\varepsilon$ with $a\in L$ and $\varepsilon \in \dot{\cao}$.
Then $\varepsilon'=a'\in L \cap \dot{\cao} = \{0\}$, so $a \in C_L$ and $\varepsilon \in C \cap \dot{\cao} = \{0\}$.
Thus $c=a\in C_L$ and so $C=C_L$.
Hence $K$ is not just an ordered differential field extension of $L$, but even an ordered valued differential field extension of $L$, since $C=C_L$ yields $\ca O \cap L = \ca O_L$.
It remains to appeal to the model completeness of $T^{\nl}_{\sm}$ in the language $\{ +, -, \cdot, 0, 1, \der, \les, \ca O \}$ \cite[Corollary~16.2.5]{adamtt}.
\end{proof}

The paper \cite{pc-preH-gap} studies $\d$-henselian pre-$H$-fields, and the results of that paper are critical to the study of transserial tame pairs.
The rest of this section explains the connection by showing how a transserial tame pair is essentially the same as a $\d$-Hensel-Liouville closed pre-$H$-field whose differential residue field is a model of $T^{\nl}_{\sm}$. This is Corollary~\ref{cor:tamepairequiv}.
The approach is to examine which properties descend from $(K,\ca O)$ to a coarsening $(K,\dot{\ca O})$ or to its differential residue field $\res(K,\dot{\ca O})$, and conversely which can be lifted from $(K,\dot{\ca O})$ and $\res(K,\dot{\ca O})$ to $(K,\ca O)$.
Note that a transserial tame pair $(K,L)$ yields the coarsening $(K,\dot{\ca O})$ of $(K,\ca O)$.
Answering these questions gives more precise results, which in particular show how a model of $T^{\nl}_{\sm}$ that is large in a certain sense, for instance $\HH$, $\No$, or a maximal Hardy field, yields a transserial tame pair.

Conversely, transserial tame pairs are fundamentally about large models of the theory of $\T$, such as $\HH$.
To see this, note that if $L \subseteq \T$ is a Liouville closed $H$-subfield of $\T$ containing $\R$, then $L$ contains $x$ and therefore $\exp_n(x)$ for every $n$. In particular, $(\T, L)$ is not a transserial tame pair.

\subsection{Coarsening}
In \cite[Section~8.A]{pc-preH-gap}, we constructed an example of a two-sorted structure to which the results of that paper apply by taking an $\aleph_0$-saturated elementary extension of $\T$ and enlarging its valuation ring by all exponentially bounded elements, elaborating on \cite[Example~10.1.7]{adamtt}.
Some technical verifications showed that this coarsening is then a $\d$-Hensel-Liouville closed pre-$H$-field whose differential residue field is a model of $T^{\nl}_{\sm}$.
In fact, by going carefully through the proof, all we actually needed was the coarsened valuation to satisfy a few axioms summarized as follows.
Let $T^{\nl,\dhl}$ be the theory of structures $(K, \ca O, \dot{\ca O})$ in the language $\{ +, -, \cdot, 0, 1, \der, \les, \ca O, \dot{\ca O} \}$ of ordered differential fields expanded by two unary predicates for valuation rings such that
\begin{enumerate}
    \item\label{Tnldhl1} $(K, \ca O) \models T^{\nl}_{\sm}$;
    \item\label{Tnldhl2} $\dot{\ca O} \supseteq \ca O$ is a convex valuation ring of $K$ with $\dot{\ca O} \neq K$;
    \item\label{Tnldhl3} for all $a \in K \setminus \dot{\ca O}$, we have $a^{\dagger} \in K \setminus \dot{\ca O}$;
    \item\label{Tnldhl4} for all (equivalently, some) $a \in K \setminus \ca O^{\x}$ with $a^{\dagger} \asymp a$, we have $a \in \dot{\ca O}^{\x}$.
\end{enumerate}
It follows from \ref{Tnldhl1} together with either \ref{Tnldhl3} or \ref{Tnldhl4} that $\ca O \neq \dot{\ca O}$.
Then \cite[Section~8.A]{pc-preH-gap} shows the following.
\begin{prop}\label{prop:coarsen-Tnldhl}
If $(K, \ca O, \dot{\ca O}) \models T^{\nl,\dhl}$, then $(K, \dot{\ca O})$ is a $\d$-Hensel-Liouville closed pre-$H$-field and the differential residue field of $(K, \dot{\ca O})$ models $T^{\nl}_{\sm}$ as a differential field.
\end{prop}
Note that it is part of the $\d$-henselianity of $(K, \dot{\ca O})$ that $(K, \dot{\ca O})$ has small derivation, so it makes sense to speak of the differential residue field of $(K, \dot{\ca O})$.

In \cite[Proposition~8.1]{pc-preH-gap}, we stated this for $K$ an $\aleph_0$-saturated elementary extension of $\T$ and $\dot{\ca O}$ the set of exponentially bounded elements of $K$, which satisfies \ref{Tnldhl3}--\ref{Tnldhl4} and $\ca O \neq \dot{\ca O}$.
The saturation was only used to ensure $\dot{\ca O} \neq K$.
This elaborated on \cite[Example~10.1.7]{adamtt}.
However, \cite[Proposition~8.1]{pc-preH-gap} and its proof were stated in terms of the value group and used the machinery of asymptotic couples, which we briefly reviewed in Section~\ref{subsec:newtonomega}.
For easier comparison with \cite[Section~8.A]{pc-preH-gap} we reformulate \ref{Tnldhl1}--\ref{Tnldhl4} in those terms now, but the reader could jump to Lemma~\ref{lem:TTPcoarsenTdhl}, especially if they do not intend to read the later proofs of Lemma~\ref{lem:uncoarsen-newtonian} and \ref{lem:uncoarsen-omegafree}.
By translation into these terms via $\Delta = v(\dot{\ca O}^{\x})$, we have $(K, \ca O, \dot{\ca O}) \models T^{\nl,\dhl}$ if and only if:
\begin{enumerate}[label=(\arabic*$^{\prime}$)]
    \item $(K, \ca O) \models T^{\nl}_{\sm}$;
    \item $\Delta \neq \{0\}$ is a convex subgroup of $\Gamma$;
    \item $\psi(\Gamma\setminus\Delta) \subseteq \Gamma\setminus\Delta$;
    \item $1 \in \psi(\Delta^{\neq})\cap\Delta$, where $1 \in \Gamma^>$ is the unique element of $\Gamma$ satisfying $\psi(1)=1$ (see \cite[Lemma~9.2.15]{adamtt}).
\end{enumerate}
Note that by \cite[Lemma~9.2.25]{adamtt}, $\psi(\Delta^{\neq})\cap\Delta \neq \0$ is equivalent to $\psi(\Delta^{\neq})\subseteq\Delta$, which is used in the proof of \cite[Proposition~8.1]{pc-preH-gap}.

Now we use Proposition~\ref{prop:coarsen-Tnldhl} to show how transserial tame pairs yield $\d$-Hensel-Liouville closed pre-$H$-fields.
\begin{lem}\label{lem:TTPcoarsenTdhl}
Let $(K, L)$ be a transserial tame pair and $\dot{\ca O} \coloneqq \conv_K(L)$.
Then $(K, \ca O, \dot{\ca O}) \models T^{\nl,\dhl}$, $(K, \dot{\ca O})$ is a $\d$-Hensel-Liouville closed pre-$H$-field, and $L$ is a lift of the differential residue field of $(K, \dot{\ca O})$.
\end{lem}
\begin{proof}
We verify that $(K, \ca O, \dot{\ca O}) \models T^{\nl,\dhl}$.
For \ref{Tnldhl2}, $\dot{\ca O} \supseteq \ca O$ was shown in the proof of Lemma~\ref{lem:elemext}, and if $\dot{\ca O} = K$, then $K=L$, a contradiction.
For \ref{Tnldhl3}, suppose towards a contradiction that $a \in K \setminus \dot{\ca O}$ but $a^{\dagger} \in \dot{\ca O}$. Then since $L$ is closed under exponential integration and $L$ is tame in $K$, we have $b \in L^{\x}$ with $a^{\dagger} \dotrel{\sim} b^{\dagger}$. But $b \dotrel{\asymp} 1$ and $a \not\dotrel{\asymp} 1$, so $a^{\dagger} \not\dotrel{\sim} b^{\dagger}$ by \cite[Corollaries~9.1.4 and 9.2.26]{adamtt}, a contradiction.
For \ref{Tnldhl4}, let $a \in K \setminus \ca O^{\x}$ with $a^{\dagger}\asymp a$. By the uniqueness in \cite[Lemma~9.2.15]{adamtt}, there exists $b \in L$ with $a \asymp b$. It follows that $a \in \dot{\ca O}^{\x}$.
\end{proof}
Note that the maximal ideal $\dot{\cao}$ of $\dot{\ca O}$ is $\dot{\cao} = \{ \varepsilon \in K : 0\les |\varepsilon|<L^> \}$ and satisfies $\der\dot{\cao} \subseteq \dot{\cao}$.
Alternatively, we can obtain Lemma~\ref{lem:TTPcoarsenTdhl} as a consequence of the following generalization.

\begin{lem}\label{lem:largeTnlcoarsendhl}
Suppose that $K \models T^{\nl}_{\sm}$, $L \prece K$, and there is $a \in K$ with $a>L$.
Let $\dot{\ca O} \coloneqq \conv_K(L)$.
Then $(K, \ca O, \dot{\ca O}) \models T^{\nl,\dhl}$, $(K, \dot{\ca O})$ is a $\d$-Hensel-Liouville closed pre-$H$-field, and there is a differential subfield $L^* \subseteq \dot{\ca O}$ such that $L \prece L^* \prece K$ and $(K, L^*)$ is a transserial tame pair.
\end{lem}
\begin{proof}
First we verify that $(K, \ca O, \dot{\ca O}) \models T^{\nl,\dhl}$.
Items \ref{Tnldhl1} and \ref{Tnldhl2} are clear from the assumptions.
Let $a \in K \setminus \dot{\ca O}$.
If $a>L$, then $a^{\dagger}>L$ by \cite[Lemma~16.6.9]{adamtt}.
If $0<a<L^>$, then by applying the lemma to $a^{-1}$ we get $a^{\dagger}<L$.
This shows \ref{Tnldhl3}.
For \ref{Tnldhl4}, let $a \in K\setminus \ca O^{\x}$ with $a^{\dagger} \asymp a$.
By \cite[Lemma~9.2.15]{adamtt} we have $b \in L^{\x}$ with $b^{\dagger} \asymp b$, so $a \asymp b$, and thus $a \in \dot{\ca O}^{\x}$.
Hence $(K, \dot{\ca O})$ is a $\d$-Hensel-Liouville closed pre-$H$-field and its differential residue field $\dot{K}\models T^{\nl}_{\sm}$ by Proposition~\ref{prop:coarsen-Tnldhl}.

Next, by Proposition~\ref{adh:7.1.3}, extend $L$ to a lift $L^* \subseteq \dot{\ca O}$ of $\dot{K}$, so $(K, L^*)$ is a transserial tame pair.
Finally, we have $L \prece L^* \prece K$ by Lemma~\ref{lem:elemext} and the model completeness of $T^{\nl}_{\sm}$ in the language $\{ +, -, \cdot, 0, 1, \der, \les, \ca O \}$.
\end{proof}

Lemma~\ref{lem:largeTnlcoarsendhl} shows that a large model of $T^{\nl}_{\sm}$ can be expanded to a transserial tame pair and also that a coarsening of this large model yields a $\d$-Hensel-Liouville closed pre-$H$-field whose differential residue field is a model of $T^{\nl}_{\sm}$. 
In particular, we obtain a claim from the introduction.
\begin{cor}\label{cor:largeTnlexTTP}
Let $K$ be $\HH$, $\No$, or a maximal Hardy field and identify $\T$ with its image under an elementary embedding $\T \to K$.
Then $\T$ can be extended to a differential subfield $\T^* \subseteq \dot{\ca O} = \conv_{K}(\T)$ so that $(K, \T^*)$ is a transserial tame pair.
\end{cor}
\begin{proof}
For the element $a \in K$ with $a>\T$ in the first two examples, one can take for instance $E_{\omega}(x) \in \HH$ and the ordinal $\epsilon_0 \in \No$.
The existence of a transexponential germ in any maximal Hardy field follows from \cite[Corollary~12.24]{bosher-neworders} and \cite[Theorem~1.3]{bosher-transexpHf}.
\end{proof}
Note that the embedding $\T \to K$ when $K$ is a maximal Hardy field is not canonical.
In Section~\ref{sec:hyperseries}, we give a more explicit example of a $\T^* \subseteq \dot{\ca O}_{\HH}$ with $(\HH, \T^*)$ a transserial tame pair. 

\subsection{Uncoarsening}
The rest of this section is dedicated to Proposition~\ref{prop:uncoarsen-dhlTnl}, which involves lifting properties of the differential residue field of a pre-$H$-field with small derivation to the ambient pre-$H$-field equipped with its natural valuation.
Ultimately, Corollary~\ref{cor:tamepairequiv} combines the results of this section to show that transserial tame pairs, models of $T^{\nl,\dhl}$, and $\d$-Hensel-Liouville closed pre-$H$-fields whose differential residue field is a model of $T^{\nl}_{\sm}$ are all equivalent.

To that end, we fix some assumptions for the rest of the section and review some terminology.
Fix an ordered differential field $K$ and a valuation ring $\dot{\ca{O}}$ of $K$ such that $(K, \dot{\ca{O}})$ is a pre-$H$-field with small derivation.
We also have the natural valuation ring $\ca O = \conv_K(C)$ of $K$.
Our goal is to deduce properties of $(K, \ca O)$ from properties of $(K, \dot{\ca O})$ and its ordered differential residue field.

First, since $(K, \dot{\ca O})$ is a pre-$H$-field, we have $C \subseteq \dot{\ca O}$ and thus $\ca O \subseteq \dot{\ca O}$, so $(K, \dot{\ca O})$ is a coarsening of $(K, \ca O)$.
We assume some familiarity with coarsening of valued fields but review the notation we use, following \cite[Section~3.4]{adamtt}, and explain how the derivation interacts with coarsening.
Let $\dot{\cao} \subseteq \cao$ be the maximal ideal of $\dot{\ca{O}}$, so the residue field of $(K, \dot{\ca O})$ is $\dot{K} \coloneqq \dot{\ca O}/\dot{\cao}$.
For $a \in \dot{\ca O}$, we set $\dot{a} \coloneqq a + \dot{\cao} \in \dot{K}$ and construe $\dot{K}$ as a valued differential field with valuation ring $\ca O_{\dot{K}} \coloneqq \{ \dot{a} : a \in \ca O \}$, whose maximal ideal 
is $\cao_{\dot{K}} \coloneqq \{ \dot{a} : a \in \cao \}$, and the induced derivation.
The derivation on $\dot{K}$ makes sense since $(K, \dot{\ca O})$ has small derivation, so the residue map $\dot{\ca O} \to \dot{K}$ is a differential ring homomorphism that satisfies $\dot{a}\ges 0$ whenever $a \ges 0$ in $\dot{\ca O}$.
It follows easily that $(\dot{K}, \ca O_{\dot{K}})$ has small derivation if and only if $(K, \ca O)$ does.
The natural ring homomorphism $\ca O \to \ca O_{\dot{K}}/\cao_{\dot{K}}$ has kernel $\cao$ and satisfies $\dot{a}+\cao_{\dot K}\ges 0$ whenever $a \ges 0$ in $\ca O$, inducing an ordered field isomorphism  $\res(K, \ca O) \to \res(\dot{K}, \ca O_{\dot{K}})$.
If $(K, \ca O)$ has small derivation, then this is an ordered \emph{differential} field isomorphism.

Since $C \subseteq \dot{\ca O}$, the residue map $\dot{\ca O} \to \dot{K}$ is injective on $C$.
Below we need the assumption that $C$ maps moreover onto $C_{\dot{K}}$, in which case we say $(K, \dot{\ca O})$ is \deft{residue constant closed}; this term was used in \cite{pc-preH-gap} but there it also assumed henselianity, which was only for convenience.
The Hardy field $\Q(\sqrt{2+x^{-1}})$, where $x$ is the germ of the identity, so $x\succ 1$, is a pre-$H$-field with small derivation.
As explained at the end of \cite[Section~9.1]{adamtt}, it is not an $H$-field, nor is it residue constant closed.
Observe in connection with this example that $(K,\dot{\ca O})$ is an $H$-field if and only if the derivation on $\dot{K}$ is trivial and $(K,\dot{\ca O})$ is residue constant closed.

It is important that if $(K, \dot{\ca O})$ is residue constant closed, then $\ca O_{\dot{K}} = \conv_{\dot{K}}(\dot{C}) = \conv_{\dot{K}}(C_{\dot{K}})$, so $\ca O_{\dot{K}}$ is the natural valuation ring of $\dot{K}$.
We use this implicitly in the following lemmas.

\begin{lem}\label{lem:uncoarsen-Hf}
Suppose that $(K, \dot{\ca{O}})$ is residue constant closed and $\dot{K}$ is an $H$-field.
Then $K$ is an $H$-field.
\end{lem}
\begin{proof}
Let $a \in \ca O$.
We need to find $c \in C$ with $a-c \in \cao$.
Since $\dot{K}$ is an $H$-field, we have $b \in \ca O$ and $\varepsilon \in \cao$ such that $\dot{a}=\dot{b}+\dot{\varepsilon}$ and $\dot{b}'=0$.
Since $(K, \dot{\ca O})$ is residue constant closed, we have $c \in C$ such that $\dot{c}=\dot{b}$.
It follows that $a-c \in \cao$.

Now let $a \in K$ with $a>C$.
We need to show that $a'>0$.
Since $(K, \dot{\ca O})$ is a pre-$H$-field, if $a>\dot{\ca O}$, then $a'>0$.
If $a \in \dot{\ca O}$, then $\dot{a} > C_{\dot{K}}=\dot{C}$, so $\dot{a}' > 0$ since $\dot{K}$ is an $H$-field, and hence $a'>0$.
Therefore $K$ is an $H$-field.
\end{proof}

In the technical next two lemmas, it is convenient to view the coarsening $(K, \dot{\ca O})$ of $(K, \ca O)$ through the lens of the value group instead of the valuation ring, which makes arguments with asymptotic couples possible.
Then letting $\Delta \coloneqq v(\dot{\ca O}^{\x})$, a convex subgroup of the value group $\Gamma$ of $(K, \ca O)$, $(K, \dot{\ca O})$ is the coarsening of $(K, \ca O)$ by $\Delta$, with valuation $\dot{v} \colon K^{\x} \to \Gamma/\Delta$.
The valuation of $(\dot{K}, \ca O_{\dot{K}})$ is $v \colon \dot{K}^{\x} \to \Delta$ defined by $v\dot{a}=va$, where $a \in \dot{\ca O}^{\x}$.

\begin{lem}\label{lem:uncoarsen-newtonian}
Suppose that $(K, \dot{\ca O})$ is $\d$-henselian and $\dot{K}$ is a newtonian $H$-field.
Then $K$ is a newtonian $H$-field.
\end{lem}
\begin{proof}
We can assume that $\Delta \neq \{0\}$.
Since $(K, \dot{\ca O})$ is $\d$-henselian, it is residue constant closed \cite[Lemmas~7.1.8 and 9.4.10]{adamtt}, so by the previous lemma $K$ is an $H$-field.
Let $(\Gamma, \psi)$ be its asymptotic couple.
It follows that the asymptotic couple of $\dot{K}$ is $(\Delta, \psi|_{\Delta^{\neq}})$, since $\ca O_{\dot{K}}$ is the natural valuation ring of $\dot{K}$.
If there is $\gamma \in \Gamma^{\neq}$ with $\max\psi(\Gamma^{\neq}) = \psi(\gamma)$, then for any $\delta\in \Delta^{\neq}$ with $0<|\delta|<|\gamma|$, we have $\psi(\gamma)=\psi(\delta)\in\Delta$, so $\psi(\delta)$ is the maximum of $\psi(\Delta^{\neq})$, contradicting that $\dot{K}$ is newtonian.
Hence $\psi(\Gamma^{\neq})$ has no maximum.

Let $P \in K\{Y\}$ with $\ndeg P = 1$ and take $\phi \in K$ with $v\phi \les \psi(\gamma)$ for some $\gamma \in \Gamma^{\neq}$ and $\ddeg P^{\phi} = \ndeg P$.
To see that $K$ is newtonian, we need to find a zero of $P$ in $\ca O$.
We can arrange by scaling $P$ that $vP=0$ and by increasing $v\phi$ that $v\phi \in \Delta$, so $\phi \dotrel{\asymp} 1$ and thus $P^{\phi} \dotrel{\asymp} P \asymp 1$ by \cite[Lemma~11.1.1]{adamtt}.
Hence $\dot{(P^{\phi})}=\dot{P}^{\dot{\phi}}$, where $\dot{P}$ and $\dot{(P^{\phi})}$ are the nonzero differential polynomials obtained by applying the residue map $\dot{\ca O} \to \dot{K}$ to the coefficients of $P$ and $P^{\phi}$, respectively, and thus $\ddeg \dot{P}^{\dot{\phi}} = \ddeg P = 1$.
It follows that $\ndeg \dot{P} = 1$, so since $\dot{K}$ is newtonian, we have $a \in \ca O$ with $\dot{P}(\dot{a})=0$.
That is, $P^{\phi}(a)=P(a) \dotrel{\prec} 1$.
Let $P^{\phi}_{+a}$ denote the differential polynomial $P^{\phi}(a+Y) \in K^{\phi}\{Y\}$.
Then $\ddeg P^{\phi}_{+a} = \ddeg P^{\phi}=1$ by \cite[Lemma~6.6.5(i)]{adamtt} and $P^{\phi}_{+a} \asymp P^{\phi} \dotrel{\asymp} 1$ by \cite[Lemma~4.5.1(i)]{adamtt},
so $ (P^{\phi}_{+a})_1 \asymp P^{\phi}_{+a} \dotrel{\asymp} 1$.
Since $(K, \dot{\ca O})$ is $\d$-henselian, we have $b \in \dot{\cao}$ with $P^{\phi}(a+b)=P(a+b)=0$.
It remains to note that $a+b \in \ca O$.
\end{proof}
Lemma~\ref{lem:uncoarsen-newtonian} is similar to \cite[Lemma~1.7.6]{adh-normalizationADA}.
The latter assumes that $(K,\ca O)$ is an $H$-asymptotic field such that $\psi(\Gamma^{\neq})$ has no greatest element, where $\ca O \neq K$ is an arbitrary valuation ring of $K$.
In contrast, Lemma~\ref{lem:uncoarsen-newtonian} does not make any assumptions on $(K,\ca O)$ except $\ca O=\conv_K(C)$, but the first paragraph of the proof establishes that $K$ satisfies the above assumptions.
Then \cite[Lemma~1.7.6]{adh-normalizationADA} yields that $K$ is newtonian.
Unlike the more restrictive setting of Lemma~\ref{lem:uncoarsen-newtonian}, \cite[Lemma~1.7.6]{adh-normalizationADA} does not require $K$ to be ordered.

\begin{lem}\label{lem:uncoarsen-omegafree}
Suppose that $(K, \dot{\ca O})$ is residue constant closed and $\dot{K}$ is an $\upomega$-free $H$-field.
Then the $H$-field $K$ is $\upomega$-free.
\end{lem}
\begin{proof}
We can assume that $\Delta\neq\{0\}$, so the asymptotic couple of $\dot{K}$ is $(\Delta, \psi|_{\Delta^{\neq}})$ as in the previous proof.
Since $\psi(\Delta^{\neq})$ has no maximum, neither does $\psi(\Gamma^{\neq})$ as before.
Let $f \in K$. We need to find $g \succ 1$ in $K$ with $f - 2(g^{\dagger\dagger})' + (g^{\dagger\dagger})^2 \succe (g^{\dagger})^2$.
First, suppose that $f \dotrel{\succ} 1$.
Take any $g \in \dot{\ca O}$ with $g \succ 1$, so $\dot{g}'>0$.
Then $g$ satisfies $g^{\dagger} \dotrel{\asymp} 1$, so $g^{\dagger\dagger} \dotrel{\prece} 1$ and $(g^{\dagger\dagger})' \dotrel{\prece} 1$.
Hence $f - 2(g^{\dagger\dagger})' + (g^{\dagger\dagger})^2 \dotrel{\sim} f \dotrel{\succ} (g^{\dagger})^2$.
Now suppose that $f \dotrel{\prece} 1$.
Since $\dot{K}$ is $\upomega$-free, we have $g \in \dot{\ca O}$ with $\dot{g} \succ 1$ and $\dot{f} - 2(\dot{g}^{\dagger\dagger})' + (\dot{g}^{\dagger\dagger})^2 \succe (\dot{g}^{\dagger})^2$, so $g \succ 1$ and $f - 2(g^{\dagger\dagger})' + (g^{\dagger\dagger})^2 \succe (g^{\dagger})^2$.
\end{proof}
Combining the previous three lemmas yields:

\begin{prop}\label{prop:uncoarsen-dhlTnl}
If $(K, \dot{\ca O})$ is a $\d$-Hensel-Liouville closed pre-$H$-field and $\dot{K} \models T^{\nl}_{\sm}$, then $K \models T^{\nl}_{\sm}$.
\end{prop}
\begin{proof}
It remains to note again that if $(K, \dot{\ca O})$ is $\d$-henselian, then it is residue constant closed by \cite[Lemmas~7.1.8 and 9.4.10]{adamtt}.
\end{proof}

The final result of the section combines the previous results, showing that we can pass between models of $T^{\nl,\dhl}$, $\d$-Hensel-Liouville closed pre-$H$-fields whose differential residue field models $T^{\nl}_{\sm}$, and transserial tame pairs. These are different perspectives on essentially the same objects.
\begin{cor}\label{cor:tamepairequiv}
Identifying $\dot{K}$ with a lift inside $\dot{\ca O}$ when necessary, the following are equivalent:
\begin{enumerate}
    \item\label{tamepairequiv1} $(K,\ca O, \dot{\ca O})\models T^{\nl,\dhl}$;
    \item\label{tamepairequiv2} $(K, \dot{\ca O})$ is a $\d$-Hensel-Liouville closed pre-$H$-field with $\dot{\ca O} \neq K$ and $\dot{K} \models T^{\nl}_{\sm}$;
    \item\label{tamepairequiv3} $(K, \dot{K})$ is a transserial tame pair.
\end{enumerate}
\end{cor}
\begin{proof}
If \ref{tamepairequiv1}, then \ref{tamepairequiv2} by Proposition~\ref{prop:coarsen-Tnldhl}.
If \ref{tamepairequiv2}, then \ref{tamepairequiv3} by Proposition~\ref{prop:uncoarsen-dhlTnl}, using Proposition~\ref{adh:7.1.3} to identify $\dot{K}$ with a lift inside $\dot{\ca O}$.
If \ref{tamepairequiv3}, then \ref{tamepairequiv1} by Lemma~\ref{lem:TTPcoarsenTdhl}.
\end{proof}

\section{Model completeness}\label{sec:modcomp}

The previous section shows how transserial tame pairs are intrinsically linked to $\d$-Hensel-Liouville closed pre-$H$-fields.
For that reason, in this section we obtain our model completeness result for transserial tame pairs as a byproduct of a more general result about pairs consisting of a $\d$-Hensel-Liouville closed pre-$H$-field and a lift of its differential residue field.
Note that the theory of $\d$-Hensel-Liouville closed pre-$H$-fields is incomplete, but becomes complete after fixing a complete theory of the differential residue field.
Likewise, it becomes model complete after fixing a model complete theory of the differential residue field.
This follows from the two-sorted results \cite[Corollaries~7.3 and 7.4]{pc-preH-gap}, but in this paper we work in a one-sorted context with a predicate for a lift of the differential residue field.

Let $(K, \dot{\ca O})$ be a $\d$-Hensel-Liouville closed pre-$H$-field.
Since $(K,\dot{\ca O})$ is $\d$-henselian, as explained in Section~\ref{sec:prelim} we can equip it with a lift $\bm k$ of its differential residue field $\res(K, \dot{\ca O})$, so $\bm k \subseteq \dot{\ca O}$ and $\bm k$ maps isomorphically as a differential field onto $\res(K, \dot{\ca O})$ under the residue map $\dot{\ca O} \to \res(K, \dot{\ca O})$.
Given any differential field $\bm k$ that is real closed, linearly surjective, and closed under exponential integration, \cite[Section~8.B]{pc-preH-gap} shows that there exists such a $(K, \dot{\ca O})$ with differential residue field (isomorphic to) $\bm k$.
For instance, $\bm k$ could be a model of $T^{\nl}_{\sm}$ as in the previous section or a closed ordered differential field in the sense of \cite{singer-codf}.
Note that these yield different completions of the theory of $\d$-Hensel-Liouville pre-$H$-fields.
In particular, if $\bm k\not\models T^{\nl}_{\sm}$, then $(K, \ca O) \not\models T^{\nl}_{\sm}$, where $\ca O=\conv_K(C)$ is the natural valuation ring of~$K$.
Recall that the theory of closed ordered differential fields is well-studied in its own right.
For example, it has quantifier elimination \cite{singer-codf}, a cell decomposition \cite{bmr-dim-codf}, and o-minimal open core \cite{point-codf}, and is distal \cite{cubidespoint-topfieldsgenericderiv,forna-kap}.

A few words on notation: We switch from $L$ to $\bm k$ here because we reserve $L$ for structures that look like $K$ in some sense. We also reserve $\ca O$ for the natural valuation of $K$, and hence continue to use $\dot{\ca O}$ for a coarsened valuation, even when $\ca O$ does not appear. Although perhaps ungainly, we hope this reduces confusion. This contrasts with \cite{pc-preH-gap}, where the natural valuation played no role and $\ca O$ was any distinguished valuation ring of~$K$.

Let $T^{\dhl}$ be the theory of $\d$-Hensel-Liouville closed pre-$H$-fields with nontrivial valuation in the language $\ca L_{\OR,\der} \cup \{ \dot{\ca O} \}$.
In this paper we work in a one-sorted context, expanding $K$ by the unary relation $\bm k$, and establish in this and the next section relative results for such structures analogous to the two-sorted results of \cite{pc-preH-gap}.
This yields the claimed model completeness of transserial tame pairs in Corollary~\ref{cor:modcompTTP}.
Additionally, we show in Theorem~\ref{thm:modcompTnldhl} that for a transserial tame pair $(K, L)$, its reduct $(K,\ca O,\dot{\ca O})$ remains model complete without $L$, where $\ca O = \conv_K(C)$ is the natural valuation of $K$ and $\dot{\ca O} = \conv_K(L)$.

Although transserial tame pairs were defined just in the language of pairs of differential fields, the theory $T^{\nl}_{\sm}$ is not model complete without the valuation ring.
Therefore, to obtain model completeness for a transserial tame pair $(K, L)$ we need to expand the language at least by the natural valuation ring of $L$, as otherwise $L$ itself would not be model complete. It turns out that we also need the convex hull of $L$ in~$K$.

Let $\ca L_{\OR, \der} \coloneqq \{ +, -, \cdot, 0, 1, \les, \der \}$ be the language of ordered differential fields and $\ca L_{\res}$ be an expansion of $\ca L_{\OR, \der}$ by predicates or function symbols on $\bm k$, which are always interpreted as trivial outside of $\bm k$.
Let $\ca L_{\lift}^{\dot{\ca O}} \coloneqq \ca L_{\res} \cup \{ \bm k, \dot{\ca O} \}$, where $\bm k$ and $\dot{\ca O}$ are unary predicates.
Fix an $\ca L_{\res}$-theory $T_{\res}$ extending the theory of ordered differential fields that are real closed, linearly surjective, and closed under exponential integration, and let $T^{\dhl}_{\lift}$ be the $\ca L_{\lift}^{\dot{\ca O}}$-theory $T^{\dhl}\cup T_{\res}$ together with axioms expressing that $\bm k$ is a differential subfield of $\dot{\ca O}$ and for every $a \dotrel{\asymp} 1$ in $K$, there exists $u \in \bm k^{\x}$ such that $a \dotrel{\sim} u$ (in other words, $\bm k$ is a lift of the differential residue field of $(K, \dot{\ca O})$).
Note that the differential residue field of a model of $T^{\dhl}$ is real closed, linearly surjective, and closed under exponential integration, so these assumptions on $T_{\res}$ are necessary.
Conversely, for any such $T_{\res}$, $T^{\dhl}_{\lift}$ is consistent by \cite[Section~8.B]{pc-preH-gap} together with Proposition~\ref{adh:7.1.3}.

Our first main result, Theorem~\ref{thm:modcompTlift}, is weaker than what we obtain in the next section, but its proof is a good outline for the proofs of Theorem~\ref{thm:modcompTnldhl} and Theorem~\ref{thm:eqthm}, which require more care.
First, we need two embedding lemmas.
\begin{lem}\label{lem:liftembed}
Suppose that $(K, \dot{\ca O})$ is a valued differential field with small derivation and $\bm k$ is a lift of $\res(K, \dot{\ca O})$.
Let $(E, \dot{\ca O}_E)$ be a valued differential subfield of $(K, \dot{\ca O})$ such that $\bm k_E \coloneqq \bm k \cap E$ is a lift of $\res(E, \dot{\ca O}_E)$.
Let $y \in \bm k \setminus \bm k_E$.
Then $(E\langle y \rangle, \dot{\ca O}_{E\langle y \rangle})$, where $\dot{\ca O}_{E\langle y \rangle} \coloneqq \dot{\ca O}\cap E\langle y \rangle$,
satisfies:
\begin{enumerate}
    \item\label{liftembed1} $(E\langle y \rangle, \dot{\ca O}_{E\langle y \rangle})$ has the same value group as $(E, \dot{\ca O}_E)$;
    \item\label{liftembed2} $\bm k_E\langle y \rangle$ is a lift of $\res(E\langle y \rangle, \dot{\ca O}_{E\langle y \rangle})$;
    \item\label{liftembed3} for any valued differential field extension $(M, \dot{\ca O}_M)$ of $(E, \dot{\ca O}_E)$ with small derivation and any lift $\bm k_M \supseteq \bm k_E$ of $\res(M, \dot{\ca O}_M)$, every differential field embedding $\bm k_E\langle y\rangle \to \bm k_M$ over $\bm k_E$ extends to a valued differential field embedding $(E\langle y\rangle, \dot{\ca O}_{E\langle y \rangle}) \to (M, \dot{\ca O}_M)$ over~$E$.
\end{enumerate}
If $K$ and $M$ are additionally ordered fields and $E$ is an ordered subfield of $K$ and $M$, and $\dot{\ca O}$ and $\dot{\ca O}_M$ are convex, then all embeddings are additionally taken to be ordered field embeddings.
\end{lem}
\begin{proof}
Fix an $(M, \dot{\ca O}_M)$ and $\bm k_M$ as in \ref{liftembed3}, and a differential field embedding $i \colon \bm k_E\langle y\rangle \to \bm k_M$ over $\bm k_E$.
First, suppose that $y$ is $\d$-algebraic over $\bm k_E$ and take a minimal annihilator $P \in \bm k_E\{Y\}^{\neq}$ of $y$ over $\bm k_E$, meaning that $P$ is irreducible, $P(y)=0$, and $Q(y)\neq 0$ for every $Q \in \bm k_E\{Y\}^{\neq}$ of strictly smaller order.
Then $P$ is also a minimal annihilator of $y$ over $E$.
Letting $z \coloneqq i(y)$, $P$ is a minimal annihilator of $z$ over $\bm k_E$, and hence over $E$, which yields a differential field isomorphism $j \colon E\langle y \rangle \to E\langle z \rangle$ over $E$ that extends $i$.
By the Zariski--Abhyankar inequality \cite{abhyankar} (see also \cite[Corollary~3.1.11]{adamtt}), both $(E\langle y \rangle, \dot{\ca O}_{E\langle y \rangle})$ and $(E\langle z \rangle, \dot{\ca O}_{E\langle z \rangle})$ have the same value group as $(E, \dot{\ca O}_E)$, where $\dot{\ca O}_{E\langle z \rangle} \coloneqq \dot{\ca O}_M \cap E\langle z \rangle$.
Then the uniqueness in \cite[Theorem~6.3.2]{adamtt} makes $j$ a valued differential field embedding $j \colon (E\langle y\rangle, \dot{\ca O}_{E\langle y \rangle}) \to (M, \dot{\ca O}_M)$.
If instead $y$ is $\d$-transcendental over $\bm k_E$, i.e., there is no $P \in \bm k_E\{Y\}^{\neq}$ with $P(y)=0$, then again we obtain a differential field isomorphism which is a valued differential field embedding by the Zariski--Abhyankar inequality and the uniqueness of \cite[Lemma~6.3.1]{adamtt}.
Property~\ref{liftembed2} is easy to check.

The additional statement in the ordered context follows from \cite[Lemma~3.1]{pc-preH-gap}.
\end{proof}

Here is the main technical embedding lemma from \cite{pc-preH-gap}, which we need a few times.
Its notation is slightly altered to fit this paper (namely, $\dot{\ca O}$ is specified, replacing the implicit $\ca O$).
Its proof also makes essential use of results from \cite{pc-dh}.
\begin{lem}[{\cite[Lemma~7.1]{pc-preH-gap}}]\label{lem:preHgap7.1}
Suppose that $(K, \dot{\ca O})$ is a $\d$-Hensel-Liouville closed pre-$H$-field, and let $(E, \dot{\ca O}_E)$ be a pre-$H$-subfield of $(K, \dot{\ca O})$ with $\res(E, \dot{\ca O}_E) = \res(K, \dot{\ca O})$.
Let $(L, \dot{\ca O}_L)$ be a $\d$-Hensel-Liouville closed pre-$H$-field such that $(L, \dot{\ca O}_L)$ is $|K|^+$-saturated.
Then any embedding $(E, \dot{\ca O}_E) \to (L, \dot{\ca O}_L)$ can be extended to an embedding $(K, \dot{\ca O}) \to (L, \dot{\ca O}_L)$.
\end{lem}

\begin{thm}\label{thm:modcompTlift}
If $T_{\res}$ is model complete in $\ca L_{\res}$, then $T^{\dhl}_{\lift}$ is model complete in~$\ca L_{\lift}^{\dot{\ca O}}$.
\end{thm}
\begin{proof}
Here and later we use a standard model completeness test via embeddings due to Robinson \cite{robinson-completetheories,robinson-modeltheoryintro}.
Let $(K, \bm k), (K^*, \bm k^*) \models T^{\dhl}_{\lift}$ such that $(K^*, \bm k^*)$ is $|K|^+$-saturated, and let $(E, \bm k_E)$ be a submodel of $(K, \bm k)$.
Given an $\ca L_{\lift}^{\dot{\ca O}}$-embedding $i \colon (E, \bm k_E) \to (K^*, \bm k^*)$, it suffices to extend $i$ to an $\ca L_{\lift}^{\dot{\ca O}}$-embedding $(K, \bm k) \to (K^*, \bm k^*)$.
Supposing that $T_{\res}$ is model complete in $\ca L_{\res}$, we can extend the $\ca L_{\res}$-embedding $i|_{\bm k_E} \colon \bm k_E \to \bm k^*$ to an $\ca L_{\res}$-embedding $j \colon \bm k \to \bm k^*$.
Given $d \in \bm k \setminus \bm k_E$ and $d^* \coloneqq j(d) \in \bm k^* \setminus i(\bm k_{E^*})$, Lemma~\ref{lem:liftembed} yields an extension of $i$ to a pre-$H$-field isomorphism $(E\langle d\rangle, \dot{\ca O}_{E\langle d\rangle}) \to (E^*\langle d^*\rangle, \dot{\ca O}^*_{E^*\langle d^*\rangle})$ sending $d$ to $d^*$.
Thus by Zorn we have a differential subfield $F$ of $K$ such that $E \cup \bm k \subseteq F$ and an extension of $i$ to an $\ca L_{\lift}^{\dot{\ca O}}$-embedding $i^* \colon (F, \bm k) \to (K^*, \bm k^*)$ with $i^*|_{\bm k} = j$.
It remains to apply Lemma~\ref{lem:preHgap7.1}.
\end{proof}

\begin{cor}\label{cor:modcompTTP}
Let $\ca L_{\res} = \ca L_{\OR, \der} \cup \{ \ca O \}$, where $\ca O$ is interpreted in a transserial tame pair $(K, L)$ as $\ca O_L = \conv_L(C)$.
The theory of transserial tame pairs is model complete in~$\ca L_{\lift}^{\dot{\ca O}}$.
\end{cor}
\begin{proof}
Recall that $T_{\res}=T^{\nl}_{\sm}$ is model complete in $\ca L_{\res}$ \cite[Corollary~16.2.5]{adamtt}.
Note also that any model of $T^{\nl}_{\sm}$ is Liouville closed and linearly surjective \cite[Corollary~14.2.2]{adamtt}, so $T^{\nl}_{\sm}$ satisfies the assumptions on~$T_{\res}$.
\end{proof}

In the next result, recall the theory $T^{\nl,\dhl}$ in the language $\ca L_{\OR,\der} \cup \{ \ca O, \dot{\ca O} \}$ from the previous section, and that we have $(K, \dot{\ca O}) \models T^{\dhl}$ whenever $(K, \ca O, \dot{\ca O}) \models T^{\nl,\dhl}$.
From future work on dimension in $\d$-henselian pre-$H$-fields along the lines of \cite{adh-dimension}, no lift of $\res(K,\dot{\ca O})$ is definable in $(K, \ca O, \dot{\ca O})$, so $(K, \ca O, \dot{\ca O})$ is a proper reduct of $(K, \bm k, \ca O, \dot{\ca O})$ for any such lift $\bm k$.
In comparison with the previous result, which interpreted $\ca O$ only as the natural valuation ring of $L$, here it is interpreted as the natural valuation ring of~$K$.
We give two proofs, first deducing the result from Corollary~\ref{cor:modcompTTP}, then presenting an alternate proof following the same outline as Theorem~\ref{thm:modcompTlift}, but with more care taken regarding the interactions of the two valuation rings, since the lifts are not in the language.

\begin{thm}\label{thm:modcompTnldhl}
The theory $T^{\nl,\dhl}$ is model complete in $\ca L_{\OR,\der} \cup \{ \ca O, \dot{\ca O} \}$.
\end{thm}
\begin{proof}[Proof 1]
First we give a shorter proof relying on the model completeness of transserial tame pairs.
Let $(K, \ca O, \dot{\ca O}) \subseteq (K^*, \ca O^*, \dot{\ca O}^*)$ be two models of $T^{\nl,\dhl}$.
Since $(K, \dot{\ca O})$ and $(K^*, \dot{\ca O}^*)$ are $\d$-henselian, by Proposition~\ref{adh:7.1.3} we can equip $(K, \dot{\ca O})$ with a lift $L \subseteq \dot{\ca O}$ of $\res(K, \dot{\ca O})$, and then extend $L$ to a lift $L^* \subseteq \dot{\ca O}^*$ of $\res(K^*, \dot{\ca O}^*)$.
Then $(K, L) \subseteq (K^*, L^*)$ are transserial tame pairs by Corollary~\ref{cor:tamepairequiv}, so we have $(K, L) \prece (K^*, L^*)$ in the language $\ca L_{\lift}^{\dot{\ca O}}$ with $\ca L_{\res} = \ca L_{\OR, \der} \cup \{ \ca O \}$.
In applying that corollary, the relation symbol $\ca O \in \ca L_{\res}$ is interpreted as the natural valuations of $L$ and $L^*$, but then the same holds interpreting $\ca O \in \ca L_{\res}$ as the natural valuations of $K$ and $K^*$, which makes $(K, \ca O, \dot{\ca O})$ and $(K^*, \ca O^*, \dot{\ca O}^*)$ reducts of $(K, L)$ and $(K^*, L^*)$, respectively.
It follows that $(K, \ca O, \dot{\ca O}) \prece (K^*, \ca O^*, \dot{\ca O}^*)$.
\end{proof}
\begin{proof}[Proof 2]
We also give a longer direct proof elaborating on the proof of Theorem~\ref{thm:modcompTlift}.
For this proof, let $\ca L \coloneqq \ca L_{\OR,\der}$, $\ca L^{\ca O} \coloneqq \ca L \cup \{\ca O\}$, $\ca L^{\dot{\ca O}} = \ca L \cup \{\dot{\ca O}\}$, and $\ca L^{\ca O, \dot{\ca O}} \coloneqq \ca L \cup \{\ca O, \dot{\ca O}\}$.

Let $(K, \ca O, \dot{\ca O}), (K^*, \ca O^*, \dot{\ca O}^*) \models T^{\nl,\dhl}$ such that $(K^*, \ca O^*, \dot{\ca O}^*)$ is $|K|^+$-saturated, and fix a submodel $(E, \ca O_E, \dot{\ca O}_E) \models T^{\nl,\dhl}$ of $(K, \ca O, \dot{\ca O})$.
Given an $\ca L^{\ca O, \dot{\ca O}}$-embedding $i \colon (E, \ca O_E, \dot{\ca O}_E) \to (K^*, \ca O^*, \dot{\ca O}^*)$, we need to extend $i$ to an $\ca L^{\ca O, \dot{\ca O}}$-embedding $(K, \ca O, \dot{\ca O}) \to (K^*, \ca O^*, \dot{\ca O}^*)$.
First, equip $E$ with a lift $\bm k_E \subseteq \dot{\ca O}_E$ of the differential residue field $\res(E, \dot{\ca O}_E)$.
In particular, $(E, \bm k_E)$ is a transserial tame pair by Corollary~\ref{cor:tamepairequiv}, so $(\bm k_E, \ca O_{\bm k_E}) \prece (E, \ca O_E)$ by Lemma~\ref{lem:elemext}.
Then $i$ restricts to an (elementary) $\ca L^{\ca O}$-embedding $i|_{\bm k_E} \colon (\bm k_E, \ca O_{\bm k_E}) \to (K^*, \ca O^*)$ with image contained in $\dot{\ca O}^*$.
Now, we can extend $\bm k_E$ to a lift $\bm k \subseteq \dot{\ca O}$ of the differential residue field $\res(K, \dot{\ca O})$; as before, this is even a lift of $\res(K, \dot{\ca O})$ as an ordered \emph{valued} differential field.

Our aim is first to extend $i|_{\bm k_E}$ to $\bm k$.
Let $d \in \bm k \setminus \bm k_E$ and consider
\[
\tp^{\bm k}(d\mid\bm k_E)\ =\ \{ \varphi(x, \overline{u}) : \varphi(x, \overline{y})\in \ca L,\ \overline{u} \in \bm k_E^{|\overline{y}|},\ \bm k \models \varphi(d,\overline{u}) \}.
\]
Note that $\tp^{\bm k}(d\mid\bm k_E) = \tp^{K}(d\mid\bm k_E)$, since $(\bm k, \ca O_{\bm k}) \prece (K, \ca O)$.
By saturation and the model completeness of $T^{\nl}_{\sm}$ we can realize the type
\[
\{ \varphi(x, i(\overline{u})) : \varphi(x, \overline{u})\in \tp^{\bm k}(d\mid\bm k_E) \} \cup \{ x \in \dot{\ca O} \}
\]
in $(K^*, \ca O^*, \dot{\ca O}^*)$ by an element $d^* \in \dot{\ca O}^*$.
This yields an $\ca L$-isomorphism $\bm k_E\langle d \rangle \to i(\bm k_E)\langle d^*\rangle$ extending $i|_{\bm k_E}$, which is moreover elementary as a partial $\ca L^{\ca O}$-map $(K, \ca O) \to (K^*, \ca O^*)$.
Thus by Zorn we get an elementary $\ca L^{\ca O}$-embedding $i^* \colon (\bm k, \ca O_{\bm k}) \to (K^*, \ca O^*)$ with image contained in~$\dot{\ca O}^*$.

Now we use Lemma~\ref{lem:liftembed} as in the proof of Theorem~\ref{thm:modcompTlift} to get a differential subfield $F$ of $K$ with $E \cup \bm k \subseteq F$ and an $\ca L^{\dot{\ca O}}$-embedding $j \colon (F, \dot{\ca O}_F) \to (K^*, \dot{\ca O}^*)$ extending $i$ and $i^*$, where $\dot{\ca O}_F = \dot{\ca O} \cap F$.
Note that $\ca O_F = \ca O \cap F$, since $C_F = C_{\bm k} = C_K$.
Next we check that $j$ is moreover an $\ca L^{\ca O, \dot{\ca O}}$-embedding.
Let $a \in F$. It is clear that if $a \in \ca O_F$, then $j(a) \in \ca O^*$, so it remains to show that if $a \notin \ca O_F$, then $j(a) \notin \ca O^*$.
First, suppose that $|a|>\bm k$. In that case, $|j(a)|>i^*(\bm k)$, so $j(a) \notin \ca O^*$.
Second, suppose that $a \in \dot{\ca O}_F \setminus \ca O_F$. Then $a = u + \varepsilon$, with $u \in \bm k$ and $\varepsilon \in \dot{\cao}_F$, so $j(a)=i^*(u)+j(\varepsilon)$. (Recall, $\dot{\cao}_F=\{ a \in F : |a|<\bm k^> \}$.) Moreover, $|u|>C_{K}$, so also $|i^*(u)|>C_{K^*}$ by the construction of $i^*$. Thus $i^*(u) \notin \ca O^*$, so $j(a) \notin \ca O^*$.
This concludes the proof that $j \colon (F, \ca O_F, \dot{\ca O}_F) \to (K^*, \ca O^*, \dot{\ca O}^*)$ is an $\ca L^{\ca O, \dot{\ca O}}$-embedding.

Finally, we appeal to Lemma~\ref{lem:preHgap7.1} to extend $j$ to an $\ca L^{\dot{\ca O}}$-embedding $(K, \dot{\ca O}) \to (K^*, \dot{\ca O}^*)$. The same reasoning as for $j$ shows that $j^*$ is even an $\ca L^{\ca O, \dot{\ca O}}$-embedding.
\end{proof}
Theorem~\ref{thm:modcompTnldhl} yields Theorem~\ref{thmint:HNomodcomp} in the introduction about the structures $(\HH, \ca O_{\HH}, \dot{\ca O}_{\HH})$, $(\No, \ca O_{\No}, \dot{\ca O}_{\No})$, and $(H, \ca O_{H}, \dot{\ca O}_H)$ for $H$ any maximal Hardy field by Lemma~\ref{lem:largeTnlcoarsendhl}.

To conclude the section, we now improve the above model completeness results to model companion results.
Let $T_{\res,0}$ be an $\ca L_{\res}$-theory, and $T^{\dhl}_{\lift,0}$ be the $\ca L_{\lift}^{\dot{\ca O}}$-theory of structures $(K, \bm k, \dot{\ca O})$ expressing that $(K, \dot{\ca O})$ is a pre-$H$-field with gap~$0$, $\bm k$ is a lift of $\res(K, \dot{\ca O})$, and $\bm k \models T_{\res,0}$.

\begin{lem}
Suppose that every model of $T_{\res,0}$ can be extended to a model of $T_{\res}$.
Then every model of $T^{\dhl}_{\lift,0}$ can be extended to a model of  $T^{\dhl}_{\lift}$.
\end{lem}
\begin{proof}
Let $(K, \bm k, \dot{\ca O}) \models T^{\dhl}_{\lift,0}$.
First, extend $\bm k$ to $\bm k_L \models T_{\res}$.
Next, using \cite[Corollary~3.4 and Theorem~6.16]{pc-preH-gap}, extend $(K, \dot{\ca O})$ to a $\d$-Hensel-Liouville closed pre-$H$-field $(L, \dot{\ca O}_L)$ so that $\res(L, \dot{\ca O}_L) \cong \bm k_L$ as ordered differential fields over $\bm k$.
Using Proposition~\ref{adh:7.1.3} to identify $\bm k_L$ with a lift inside $\dot{\ca O}_L$ extending $\bm k$ yields an extension $(L, \bm k_L, \dot{\ca O}_L) \models T^{\dhl}_{\lift}$ of $(K, \bm k, \dot{\ca O})$.
\end{proof}

Combining this with Theorem~\ref{thm:modcompTlift} yields:
\begin{cor}
If $T_{\res}$ is the model companion of $T_{\res,0}$, then the $\ca L_{\lift}^{\dot{\ca O}}$-theory $T^{\dhl}_{\lift}$ is the model companion of $T^{\dhl}_{\lift,0}$.
\end{cor}

Let $T_{\pair,0}$ be the theory of structures $(K, L, \ca O, \dot{\ca O})$ in the language $\ca L_{\OR,\der} \cup \{ \bm k, \ca O, \dot{\ca O} \}$ such that
\begin{enumerate}
    \item $(K, \dot{\ca O})$ is a pre-$H$-field with gap~$0$;
    \item $L=\bm k$ is a lift of $\res(K,\dot{\ca O})$;
    \item $L$ is an $H$-field with small derivation and $\ca O = \conv_L(C_L)$.
\end{enumerate}
Then by Corollary~\ref{cor:tamepairequiv} and the fact that $T^{\nl}_{\sm}$ is the model companion of the theory of $H$-fields with small derivation (see \cite[Introduction and Overview]{adamtt}):
\begin{cor}
The theory of transserial tame pairs is the model companion of $T_{\pair,0}$.
\end{cor}

Let $T^{\nl,\dhl}_0$ be the theory of structures $(K, \ca O, \dot{\ca O})$ in the language $\{ +, -, \cdot, 0, 1, \der, \les, \ca O, \dot{\ca O} \}$ such that
\begin{enumerate}
    \item $(K, \ca O)$ is an $H$-field with small derivation;
    \item $\dot{\ca O} \supsetneq \ca O$ is a convex valuation ring of $K$ with $\dot{\ca O} \neq K$;
    \item for all $a \in K \setminus \dot{\ca O}$, we have $a^{\dagger} \in K \setminus \dot{\ca O}$;
    \item for all $a \in \dot{\ca O}^{\x} \setminus \ca O^{\x}$, we have $a^{\dagger} \in \dot{\ca O}^{\x}$.
\end{enumerate}

\begin{lem}\label{lem:modextTnldhl}
Every model of $T^{\nl,\dhl}_0$ can be extended to a model of $T^{\nl,\dhl}$.
\end{lem}
\begin{proof}
Let $(K, \ca O, \dot{\ca O}) \models T^{\nl,\dhl}_0$.
First, $\res(K, \dot{\ca O})$ is an $H$-field with small derivation, so we can extend it to $\bm k_L \models T^{\nl}_{\sm}$.
Note that $\bm k_L$ is linearly surjective by \cite[Corollary~14.2.2]{adamtt}.
Next, $(K, \dot{\ca O})$ is a pre-$H$-field with gap~$0$, so by \cite[Corollary~3.4 and Theorem~6.16]{pc-preH-gap}, we can extend $(K, \dot{\ca O})$ to a $\d$-Hensel-Liouville closed pre-$H$-field $(L, \dot{\ca O}_L)$ so that $\res(L, \dot{\ca O}_L) \cong \bm k_L$ as ordered differential fields.
It remains to apply Corollary~\ref{cor:tamepairequiv}.
\end{proof}

Combining this with Theorem~\ref{thm:modcompTnldhl} yields:
\begin{cor}\label{cor:modcompanTnldhl}
The theory $T^{\nl,\dhl}$ is the model companion of $T^{\nl,\dhl}_0$.
\end{cor}

\section{Ax--Kochen/Ershov, relative quantifier elimination, and related theorems}\label{sec:AKErelQE}

\subsection{Introduction}
In this section, we refine the proofs in the previous section to obtain more precise relative results, ultimately achieving relative quantifier elimination with a standard part map.
These techniques also yield the completeness of transserial tame pairs and the stable embeddedness mentioned in the introduction.

As before, let $\ca L_{\OR, \der} \coloneqq \{ +, -, \cdot, 0, 1, \les, \der \}$ be the language of ordered differential fields and $\ca L_{\res}$ be an expansion of $\ca L_{\OR, \der}$ by predicates or function symbols on $\bm k$, which are always interpreted as trivial outside of $\bm k$.
For model completeness in the previous section, we expanded this language by the unary predicate $\dot{\ca O}$ for a coarsened valuation ring.
For the results of this section we need the binary predicate $\dotrel{\prece}$, so let $\ca L_{\lift}^{\dotrel{\prece}} \coloneqq \ca L_{\res} \cup \{ \bm k, \dotrel{\prece} \}$, where $\bm k$ is a unary predicate and $\dotrel{\prece}$ is the binary asymptotic relation coming from $\dot{\ca O}$.
Fix an $\ca L_{\res}$-theory $T_{\res}$ extending the theory of ordered differential fields that are real closed, linearly surjective, and closed under exponential integration, and let $T^{\dhl}_{\lift}$ be the $\ca L_{\lift}^{\dotrel{\prece}}$-theory $T^{\dhl}\cup T_{\res}$ together with axioms expressing that $\bm k$ is a lift of the differential residue field of $(K, \dot{\ca O})$ as before.

We establish an Ax--Kochen/Ershov theorem and related results for models of $T^{\dhl}_{\lift}$ along the lines of \cite[Section~7.A]{pc-preH-gap}, which contains similar results in a two-sorted setting.
The proofs are similar but take into account the extra subtleties of having the differential residue field as a lift instead of as a separate sort, and consequently some identical parts are left out.
Some aspects are also adapted from \cite{bhardwajvddries-AKElifts}.

\subsection{Equivalence theorem}
Let $\bm K = (K, \bm k)$ and $\bm K^* = (K^*, \bm k^*)$ be models of $T^{\dhl}_{\lift}$.
We aim to construct a back-and-forth system from $\bm K$ to $\bm K^*$ when $\bm K$ and $\bm K^*$ are sufficiently saturated.
To that end, a \deft{good substructure} of $\bm K$ is an $\ca L_{\lift}^{\dotrel{\prece}}$-substructure $\bm E = (E, \bm k_{E})$ of $\bm K$ such that $\bm E$ is a field and $\bm k_E = \bm k \cap E$ is a lift of $\res(E, \dot{\ca O}_E)$.
In particular, $(E, \dot{\ca O}_E)$ is a pre-$H$-subfield of $(K, \dot{\ca O})$ and $\bm k_E$ is an $\ca L_{\res}$-substructure of $\bm k$.
Let $\bm E$ and $\bm E^*$ be good substructures of $\bm K$ and $\bm K^*$, respectively.
A \deft{good map} from $\bm E$ to $\bm E^*$ is an $\ca L_{\lift}^{\dotrel{\prece}}$-isomorphism $f \colon \bm E \to \bm E^*$ such that $f|_{\bm k_E}$ is elementary as a partial $\ca L_{\res}$-map from $\bm k$ to~$\bm k^*$.

To be careful, this last condition needs a little explaining, namely how we relativize formulas to $\bm k$.
That is, we define $\r$-relative formulas by recursion so that:
\begin{itemize}
    \item every quantifier-free $\ca L_{\res}$-formula is an $\r$-relative formula;
    \item if $\varphi$ and $\theta$ are $\r$-relative formulas, then so are $\neg\varphi$, $\varphi\wedge\theta$, and $\varphi\vee\theta$;
    \item if $\varphi$ is $\r$-relative and $u$ is a variable, then $\exists u (u \in \bm k \wedge \varphi)$ and $\forall u (u \in \bm k \to \varphi)$ are $\r$-relative formulas.
\end{itemize}
In particular, an $\r$-relative formula does not involve the valuation $\dotrel{\prece}$.
By the above definition, the following is clear, and allows us to see the $\ca L_{\res}$-structure $\bm k$ inside $\bm K$, which we now do sometimes without comment.
\begin{lem}\label{lem:rrel}
Let $\varphi(x)$ be an $\ca L_{\res}$-formula, where $x=(x_1, \dots, x_n)$ is a tuple of pairwise distinct variables.
Then there is an $\r$-relative formula $\theta(x)$ such that for all $d \in \bm k^n$:
\[
\bm k \models \varphi(d)\quad \iff\quad \bm K \models \theta(d).
\]
\end{lem}

For $d \in \bm k^n$, we let $\tp_{\r}^{\bm K}(d\mid\bm k_E)$ denote the $\r$-relative type of $d$ over $\bm k_E$ in $\bm K$.
By Lemma~\ref{lem:rrel}, this is equivalent to considering the $\ca L_{\res}$-type of $d$ in~$\bm k$.

\begin{lem}\label{lem:goodresext}
Let $\bm E$ be a common good substructure of $\bm K$ and $\bm K^*$.
Suppose that $\bm k_{F}$ and $\bm k_{F^*}$ are $\ca L_{\res}$-substructures of $\bm k$ and $\bm k^*$, respectively, and let $f \colon \bm k_{F} \to \bm k_{F^*}$ be an $\ca L_{\res}$-isomorphism over $\bm k_E$.
Then there exist good substructures $\bm F = (F,\bm k_{F}) \subseteq \bm K$ and $\bm F^* = (F^*,\bm k_{F^*}) \subseteq \bm K^*$ and an $\ca L_{\lift}^{\dotrel{\prece}}$-isomorphism $g \colon \bm F \to \bm F^*$ such that:
\begin{enumerate}
    \item $\bm F$ and $\bm F^*$ extend $\bm E$;
    \item $(F,\dot{\ca O}_F)$ and $(F^*,\dot{\ca O}^*_F)$ have the same value group as $(E,\dot{\ca O}_E)$;
    \item $g$ extends $f$ and the identity on $\bm E$.
\end{enumerate}
\end{lem}
\begin{proof}
Given $d \in \bm k_F \setminus \bm k_E$ and $d^* \coloneqq f(d) \in \bm k_{F^*} \setminus \bm k_{E}$, use Lemma~\ref{lem:liftembed} to get a pre-$H$-field isomorphism $g \colon (E\langle d\rangle, \dot{\ca O}_{E\langle d\rangle}) \to (E\langle d^*\rangle, \dot{\ca O}^*_{E\langle d^*\rangle})$ over $E$ with $g(d)=d^*$.
Then $(E\langle d\rangle, \dot{\ca O}_{E\langle d\rangle})$ has the same value group as $(E,\dot{\ca O}_E)$ and $\bm k_E\langle d\rangle$ is a lift of $\res(E\langle d\rangle, \dot{\ca O}_{E\langle d\rangle})$, where $\bm k_E\langle d\rangle$ is the differential subfield of $\bm k_F$ generated by $\bm k_E$ and $d$, and $g|_{\bm k_E\langle d\rangle}=f|_{\bm k_E\langle d\rangle}$.
Although $\bm k_E\langle d\rangle$ might not be an $\ca L_{\res}$-structure of $\bm k$, so $(E\langle d\rangle, \bm k_E\langle d\rangle)$ might not be an $\ca L_{\lift}^{\dotrel{\prece}}$-substructure of $\bm K$,
Zorn yields the requisite extensions $\bm F$ and $\bm F^*$, and $\ca L_{\lift}^{\dotrel{\prece}}$-isomorphism $g \colon \bm F \to \bm F^*$.
\end{proof}

\begin{cor}\label{cor:goodresext}
Let $\bm E$ be a common good substructure of $\bm K$ and $\bm K^*$.
Suppose that $d \in \bm k^n$ and $d^* \in \bm k^n$ satisfy 
$\tp_{\r}^{\bm K}(d\mid\bm k_E)=\tp_{\r}^{\bm K^*}(d^*\mid\bm k_E)$.
Then there exist good substructures $\bm F \subseteq \bm K$ and $\bm F^* \subseteq \bm K^*$ and a good map $f \colon \bm F \to \bm F^*$ such that:
\begin{enumerate}
    \item $d \in \bm k_F^n$ and $d^* \in \bm k_{F^*}^n$;
    \item $\bm F$ and $\bm F^*$ extend $\bm E$;
    \item $(F,\dot{\ca O}_F)$ and $(F^*,\dot{\ca O}^*_F)$ have the same value group as $(E,\dot{\ca O}_E)$;
    \item $f$ extends the identity on $\bm E$;
    \item $f(d)=d^*$.
\end{enumerate}
\end{cor}
\begin{proof}
Let $\bm k_F$ be the $\ca L_{\res}$-substructure of $\bm k$ generated by $\bm k_E$ and $d$ and $\bm k_{F^*}$ be the $\ca L_{\res}$-substructure of $\bm k^*$ generated by $\bm k_E$ and $d^*$.
The assumption $\tp_{\r}^{\bm K}(d\mid\bm k_E)=\tp_{\r}^{\bm K^*}(d^*\mid\bm k_E)$ yields an extension of the identity map on $\bm k_E$ to a map $f \colon \bm k_{F} \to \bm k_{F^*}$ that is elementary as a partial $\ca L_{\res}$-map from $\bm k$ to $\bm k^*$ with $f(d)=d^*$.
It remains to apply Lemma~\ref{lem:goodresext}.
\end{proof}

The next theorem underpins all the remaining results.
\begin{thm}[Equivalence Theorem]\label{thm:eqthm}
Every good map $\bm E \to \bm E^*$ between good substructures $\bm E$ and $\bm E^*$ is elementary as a partial map from $\bm K$ to~$\bm K^*$.
\end{thm}
\begin{proof}
Let $f$ be a good map from $\bm E$ to $\bm E^*$.
Let $\kappa$ be a cardinal of uncountable cofinality such that $\max\{|E|, |\ca L_{\res}|\}<\kappa$ and $2^\lambda < \kappa$ for every cardinal $\lambda<\kappa$. 
By passing to elementary extensions, we may suppose that $\bm K$ and $\bm K^*$ are $\kappa$-saturated.
We say a good substructure $(E_1, \bm k_1)$ of $\bm K$ is \deft{small} if $|E_1|<\kappa$.
Note that if a good substructure  $(E_1,\bm k_1)$ of $\bm K$ with value group $\Gamma_1$ satisfies $\max\{|\bm k_1|,|\Gamma_1|\}<\kappa$, then $(E_1,\bm k_1)$ is small since $|E_1|\les|\bm k_1|^{|\Gamma_1|}<\kappa$ (see for example \cite[Lemma~2.2.1]{adamtt} for the first inequality).
In particular, if $(E_1, \bm k_1)$ is small and $(E_2, \dot{\ca O}_{E_2})$ is a pre-$H$-subfield of $(K,\dot{\ca O})$ that is an immediate extension of $(E_1, \dot{\ca O}_{E_1})$,
then $(E_2, \bm k_1)$ remains a small good substructure of $\bm K$.
To establish the theorem, we adapt the proof of \cite[Theorem~7.2]{pc-preH-gap}, a two-sorted analogue, to show that the set of good maps between small good substructures is a back-and-forth system from $\bm K$ to $\bm K^*$.
For this, given $a \in K \setminus E$, we need to extend $f$ to a good map with small domain containing~$a$.

We first handle the case that $a \in \bm k$.
By the saturation assumption, we have $a^* \in \bm k^*$ such that $\tp_{\r}^{\bm K}(a\mid\bm k_E)=\tp_{\r}^{\bm K^*}(a^*\mid f(\bm k_E))$.
Then Corollary~\ref{cor:goodresext} yields an extension of $f$ to a good map $\bm F \to \bm F^*$ with $a \in \bm k_F$ and $a^* \in \bm k_{F^*}$ and $a \mapsto a^*$.
Note that taking $\bm k_F$ and $\bm k_{F^*}$ as in the proof of Corollary~\ref{cor:goodresext}, we have $|\bm k_F|=|\bm k_{F^*}|<\kappa$, since $\max\{|\bm k_E|, |\ca L_{\res}|\}<\kappa$.
Thus we ensure that $\bm F$ and $\bm F^*$ are small.

With the case $a \in \bm k$ taken care of, the argument is now the same as in the proof of \cite[Theorem~7.2]{pc-preH-gap}, which we sketch.
Suppose that $a \notin \bm k$.
By iterating the above argument, we arrange that $\bm k_E$ is real closed, linearly surjective, and closed under exponential integration, and that $(E, \bm k_E) \subseteq (E\langle a \rangle, \bm k_E)$ are small good substructures of $\bm K$; a key point is that $\bm k_E$ is now a lift of both differential residue fields.
It remains to take the $\d$-Hensel-Liouville closure of $(E\langle a\rangle, \dot{\ca O}_{E\langle a\rangle})$ by \cite[Theorem~6.16]{pc-preH-gap} and apply Lemma~\ref{lem:preHgap7.1}.
\end{proof}

We now obtain the following two consequences: relative completeness and relative model completeness.
The former is often called an Ax--Kochen/Ershov theorem, and it also yields the completeness of transserial tame pairs.
\begin{cor}\label{cor:AKE}
We have $\bm K \equiv \bm K^*$ if and only if $\bm k \equiv \bm k^*$.
\end{cor}
\begin{proof}
The left-to-right direction follows from Lemma~\ref{lem:rrel}.
For the converse, suppose that $\bm k \equiv \bm k^*$.
We may assume that $\ca L_{\res}$ is an expansion of $\ca L_{\OR,\der}$ by relation symbols, so then we can identify $\Q$ with an $\ca L_{\res}$-substructure of $\bm k$ and an $\ca L_{\res}$-substructure of $\bm k^*$, respectively, and by assumption, these are $\ca L_{\res}$-isomorphic.
Consider the $\ca L_{\lift}^{\dotrel{\prece}}$-structure $(\Q, \Q)$, where $\Q$ is equipped with its usual ordered ring structure and the trivial derivation and valuation, the predicate $\bm k$ is interpreted as $\Q$ and construed as an $\ca L_{\res}$-structure as explained above.
This structure embeds into both $\bm K$ and $\bm K^*$, inducing an obvious good map between good substructures of $\bm K$ and $\bm K^*$, respectively, which is elementary as a partial map $\bm K \to \bm K^*$.
Hence $\bm K \equiv \bm K^*$.
\end{proof}

\begin{cor}\label{cor:TTPcomp}
The theory of transserial tame pairs is complete.
The theory $T^{\nl,\dhl}$ is complete.
\end{cor}
\begin{proof}
Note that $T^{\nl}_{\sm}$ is complete \cite[Corollary~16.6.3]{adamtt}.
The second statement follows from the first as in the first proof of Theorem~\ref{thm:modcompTnldhl}.
\end{proof}

\begin{cor}\label{cor:relmodcomp}
Let $\bm E = (E, \bm k_{E}) \subseteq \bm K$ with $\bm E \models T^{\dhl}_{\lift}$.
If $\bm k_E \prece \bm k$, then $\bm E \prece \bm K$.
\end{cor}
\begin{proof}
View the identity map on $\bm E$ as a map from $\bm E$ to~$\bm K$ and note that it is good.
\end{proof}
Note that Corollary~\ref{cor:relmodcomp} generalizes Theorem~\ref{thm:modcompTlift}.
Next we deduce that the differential residue field is purely stably embedded in the pair.
\begin{cor}\label{cor:kstabembed}
Any subset of $\bm k^n$ definable in the $\ca L_{\lift}^{\dotrel{\prece}}$-structure $\bm K$ \textnormal{(}with parameters from $K$\textnormal{)} is definable in the $\ca L_{\res}$-structure $\bm k$ \textnormal{(}with parameters from $\bm k$\textnormal{)}.
\end{cor}
\begin{proof}
Let $\bm L=(L, \bm k_L)$ and $\bm L^*=(L^*, \bm k_{L^*})$ be elementary $\ca L_{\lift}^{\dotrel{\prece}}$-extensions of $(K, \bm k)$ with $d \in \bm k_L^n$ and $d^* \in \bm k_{L^*}^n$ such that $\tp_{\r}^{\bm L}(d\mid\bm k)=\tp_{\r}^{\bm L^*}(d^*\mid\bm k)$.
Use Corollary~\ref{cor:goodresext} to extend the identity map on $\bm K$ to a good map from a good substructure of $\bm L$ containing $d$ to a good substructure of $\bm L^*$ containing $d^*$ such that $d \mapsto d^*$.
Then Theorem~\ref{thm:eqthm} yields $\tp^{\bm L}(d\mid K)=\tp^{\bm L^*}(d^*\mid K)$, where $\tp^{\bm L}(d\mid K)$ is the $\ca L_{\lift}^{\dotrel{\prece}}$-type of $d$ over $K$ in $\bm L$ and likewise for $\tp^{\bm L^*}(d^*\mid K)$.
The statement of the corollary follows by the Stone representation theorem.
\end{proof}

\begin{cor}\label{cor:Cstabembed}
Let $(K, L)$ be a transserial tame pair. Then:
\begin{enumerate}
    \item Any subset of $L^n$ definable in $(K, L)$ \textnormal{(}with parameters from $K$\textnormal{)} is definable in the differential field $L$ \textnormal{(}with parameters from~$L$\textnormal{)}.
    \item Any subset of $C^n$ definable in $(K, L)$ \textnormal{(}with parameters from $K$\textnormal{)} is definable in the field $C$ \textnormal{(}with parameters from~$C$\textnormal{)}.
\end{enumerate}
\end{cor}
\begin{proof}
The statement about $L$ is immediate from Corollary~\ref{cor:kstabembed} with $\ca L_{\res} = \ca L_{\OR,\der}$.
For $C$, recall that $C \subseteq L$ by Lemma~\ref{lem:elemext}, so a subset of $C^n$ definable in $(K, L)$ is definable in $L$.
The result follows from the pure stable embeddedness of $C$ in $L$, \cite[Proposition~16.6.7]{adamtt}.
\end{proof}

\subsection{Relative quantifier elimination}

In this subsection we eliminate quantifiers down to the lift of the differential residue field construed as an $\ca L_{\res}$-structure.
In the result, we make more explicit use of the $\r$-relative formulas defined before Lemma~\ref{lem:rrel}.
Additionally, we further expand the language by a binary version of the standard part map.
To explain this, fix $(K, \bm k) \models T^{\dhl}_{\lift}$.
For each $a \in \dot{\ca O}^{\x}$, the $u \in \bm k^{\x}$ with $a \dotrel{\sim} u$ is unique.
Hence we have a definable map $\pi \colon K^2 \to \bm k$ defined by, for $a, b \in K$,
\[
\pi(a,b)\ =\ 
\begin{cases}
u & \text{if}\ a \dotrel{\sim} ub,\ u \in \bm k^{\x}, \\
0 & \text{if}\ a \not\dotrel{\asymp} b\ \text{or}\ a=b=0.
\end{cases}
\]
If we were to include a function symbol for multiplicative inversion in the language, we would need only the unary version of $\pi$ (and only $\dot{\ca O}$ instead of $\dotrel{\prece}$).
Let $\ca L_{\lift}^{\dotrel{\prece},\pi} \coloneqq \ca L_{\lift}^{\dotrel{\prece}} \cup \{ \pi \}$, and continue to denote by $T^{\dhl}_{\lift}$ its natural expansion by definition to an $\ca L_{\lift}^{\dotrel{\prece},\pi}$-theory.
To state the relative quantifier elimination result precisely, we also need to define a special kind of formula.
Let $x=(x_1, \dots, x_n)$ be a tuple of pairwise distinct variables.
We call an $\ca L_{\lift}\cup\{\pi\}$-formula $\varphi(x)$ \deft{special} if $\varphi(x)$ is
\[
\varphi_{\r}\big(\pi(\sigma_1(x), \tau_1(x)), \dots, \pi(\sigma_k(x), \tau_k(x))\big)
\]
for some $k \in \N$, $\r$-relative formula $\varphi_{\r}(u_1, \dots, u_k)$, and $\ca L_{\res}^{\pi}$-terms $\sigma_1(x)$, $\tau_1(x)$, \dots, $\sigma_k(x)$, $\tau_k(x)$, where $\ca L_{\res}^{\pi} \coloneqq \ca L_{\res} \cup \{ \pi \}$.

\begin{thm}\label{thm:relQE}
Let $x$ be as above.
If $\varphi(x)$ is an $\ca L_{\lift}^{\dotrel{\prece},\pi}$-formula, then $\varphi(x)$ is $T^{\dhl}_{\lift}$-equivalent to
\begin{equation}\tag{$*$}\label{relQEform}
\big(\theta_1(x) \wedge \varphi_1(x)\big) \vee \dots \vee \big(\theta_N(x) \wedge \varphi_N(x)\big)
\end{equation}
for some $N \in \N$, quantifier-free $\ca L_{\lift}^{\dotrel{\prece}, \pi}$-formulas $\theta_1(x)$, \dots, $\theta_N(x)$, and special formulas $\varphi_1(x)$, \dots, $\varphi_N(x)$.
\end{thm}
\begin{proof}
Let $\Theta(x)$ be the set of $\ca L_{\lift}^{\dotrel{\prece},\pi}$-formulas displayed in \eqref{relQEform}.
Then $\Theta(x)$ is obviously closed under disjunction and also closed under negation, up to logical equivalence.
It suffices to show that every $x$-type consistent with $T^{\dhl}_{\lift}$ is determined by its intersection with $\Theta(x)$.
Below, $\theta(x)$ ranges over quantifier-free $\ca L_{\lift}^{\dotrel{\prece}, \pi}$-formulas and $\varphi(x)$ ranges over special formulas.
For a model $\bm K = (K, \bm k; \pi)$ of $T^{\dhl}_{\lift}$ and $a \in K^{n}$, we set
\[
\qftp^{\bm K}(a)\ \coloneqq\ \{ \theta(x) : \bm K \models \theta(a) \}
\]
and
\[
\tp_{\sp}^{\bm K}(a)\ \coloneqq\ \{ \varphi(x) : \bm K \models \varphi(a) \}.
\]
Let $\bm K = (K, \bm k; \pi)$ and $\bm K^* = (K^*, \bm k^*; \pi^*)$ be models of $T^{\dhl}_{\lift}$ and $a \in K^{n}$ and $a^* \in (K^*)^{n}$ satisfy $\qftp^{\bm K}(a) = \qftp^{\bm K^*}(a^*)$ and $\tp_{\sp}^{\bm K}(a) = \tp_{\sp}^{\bm K^*}(a^*)$.
We need to show that $\tp^{\bm K}(a) = \tp^{\bm K^*}(a^*)$.

Let $E$ be the $\ca L_{\res}^{\pi}$-structure generated by $a$ inside $\bm K$, so $\bm k_E \coloneqq \bm k \cap E = \pi(E,E)$ is an $\ca L_{\res}$-substructure of $\bm k$.
Defining $(E^*, \bm k_{E^*}^*)$ in $\bm K^*$ likewise, the assumption $\qftp^{\bm K}(a) = \qftp^{\bm K^*}(a^*)$ gives an $\ca L_{\lift}^{\dotrel{\prece},\pi}$-isomorphism $f \colon (E, \bm k_E) \to (E^*, \bm k^*_{E^*})$.
By taking fraction fields, we can arrange that $E$ and $E^*$ are fields, without changing $\bm k_E$ or $\bm k^*_{E^*}$. 
Hence $\bm E \coloneqq (E, \bm k_E)$ is a good substructure of $\bm K$ and $\bm E^* \coloneqq (E^*, \bm k_{E^*}^*)$ is a good substructure of $\bm K^*$.
Since $\tp_{\sp}^{\bm K}(a) = \tp_{\sp}^{\bm K^*}(a^*)$, $f|_{\bm k_E} \colon \bm k_E \to \bm k_{E^*}^*$ is elementary as a partial $\ca L_{\res}$-map $\bm k \to \bm k^*$.
Thus $f$ is a good map and so $\tp^{\bm K}(a) = \tp^{\bm K^*}(a^*)$ by Theorem~\ref{thm:eqthm}.
\end{proof}

Again using Lemma~\ref{lem:rrel}, we have the following special cases.
\begin{cor}
If $T_{\res}$ has quantifier elimination in $\ca L_{\res}$, then $T^{\dhl}_{\lift}$ has quantifier elimination in~$\ca L_{\lift}^{\dotrel{\prece},\pi}$.
\end{cor}

\begin{cor}
If $\ca L_{\res} = \ca L_{\OR,\der} \cup \{ \prece, \Uplambda_2, \Upomega_2 \}$, where $\Uplambda_2$, $\Upomega_2$ are the binary predicates from \cite[Chapter~16]{adamtt}, then the theory of transserial tame pairs has quantifier elimination in~$\ca L_{\lift}^{\dotrel{\prece},\pi}$.
\end{cor}
\begin{proof}
The expansion of $T^{\nl}_{\sm}$ by definitions for $\Uplambda_2$ and $\Upomega_2$ has quantifier elimination by \cite[Theorem~16.0.1]{adamtt} and the comments immediately afterwards and in the ``Notes and comments'' immediately before \cite[Section~16.1]{adamtt}.
\end{proof}

\begin{cor}
The $\ca L_{\lift}^{\dotrel{\prece},\pi}$-theory of transserial tame pairs is the model completion of the natural expansion of $T_{\pair,0}$ to an $\ca L_{\lift}^{\dotrel{\prece},\pi}$-theory, where $\ca L_{\res} = \ca L_{\OR,\der} \cup \{ \prece, \Uplambda_2, \Upomega_2 \}$.
\end{cor}

\begin{thm}\label{thm:TTPdistal}
The $\ca L_{\lift}^{\dotrel{\prece},\pi}$-theory of transserial tame pairs is distal \textnormal{(}hence has NIP\textnormal{)}, where $\ca L_{\res} = \ca L_{\OR,\der} \cup \{ \prece, \Uplambda_2, \Upomega_2 \}$.
\end{thm}
\begin{proof}
Let $(K, L)$ be a transserial tame pair.
As a tame pair of real closed closed fields, it is distal by \cite{hieronyminell}.
We are interested in $(K, L)$ as a structure in the language $\ca L_{\lift}^{\dotrel{\prece},\pi}$, with $\ca L_{\res} = \ca L_{\OR,\der} \cup \{ \prece, \Uplambda_2, \Upomega_2 \}$.
Without the derivation, this is an expansion by the definable sets $\dot{\prece}$ and $\prece$ and the two externally definable sets $\Uplambda_2$ and $\Upomega_2$, so is distal. 
It remains to use \cite[Proposition~7.1]{acgz-distal}, whose condition (3) is easy to check using the definition of a derivation, the quotient rule, and the fact that $\der$ commutes with the unary version of~$\pi$.
\end{proof}
Does distality or NIP always transfer from the $\ca L_{\res}$-theory $T_{\res}$ to the $\ca L_{\lift}^{\dotrel{\prece},\pi}$-theory $T^{\dhl}_{\lift}$?
If $T_{\res}$ has NIP, then by combining results from \cite{jahnkesimon-NIPhenselian} with well-known results for valued fields (without derivations), one can show that special formulas have NIP.
Hence by Theorem~\ref{thm:relQE}, the problem is reduced to showing that the quantifier-free $\ca L_{\lift}^{\dotrel{\prece},\pi}$-formulas have NIP.

\section{The case of hyperseries}\label{sec:hyperseries}

As mentioned in the introduction, the differential field $\HH$ of hyperseries was constructed as a field in \cite{hyperseries}, building on \cite{loghyperseries}, and as an elementary differential field extension of $\T$ in \cite{hyperseries-deriv}.
Elements of $\T$ and $\HH$ are formal series $\sum_{\mf m} f_{\mf m}\mf m$ with coefficients $f_{\mf m}$ in $\R$ and monomials $\mf m$ ranging over certain monomial groups; below, \emph{support} is meant in the sense of formal series.
Let $\On$ denote the class of ordinals.
Then $\HH$ contains for each $\alpha \in \On$ elements $E_{\alpha}(x)$ and $L_{\alpha}(x)$ to be viewed as $\alpha$ iterates of the exponential and logarithm, respectively.
In particular, $\HH$ is a proper class, although the support of every element is a set, and contains formally transexponential elements such as $E_{\omega}(x)$.
These elements allow one to solve functional equations such as $y(x+1)=\exp(y(x))$ in $\HH$ that cannot be solved in $\T$.
Nevertheless, as a first-order structure in the language $\ca L_{\OR,\der}^{\ca O}$, it is an elementary extension of~$\T$.

The class $\dot{\ca O} \coloneqq \conv_{\HH}(\T)$ consists exactly of those elements of $\HH$ bounded in absolute value by $\exp_n(x)$ for some $n$.
As explained in Section~\ref{sec:coarsenuncoarsen}, with $\ca O \coloneqq \conv_{\HH}(\R)$ and $\dot{\ca O}$, the structure $(\HH, \ca O, \dot{\ca O}) \models T^{\nl,\dhl}$ and $\T$ can be elementarily extended to $\T^* \subseteq \dot{\ca O}$ so that $(\HH, \T^*)$ is a transserial tame pair; if desired, we can take $\T^*$ to contain the differential field $\LL$ of logarithmic hyperseries from \cite{loghyperseries}.
In particular, all the results of Sections~\ref{sec:modcomp} and \ref{sec:AKErelQE} apply to the pair $(\HH, \T^*)$.
But the existence of such a $\T^*$ is obtained by Zorn, so in this section we provide a more explicit example of a $\T^*$ that works.

Let $\mf{M}$ be the monomial group of $\HH$ and $\mf{B} \coloneqq \mf{M}\cap\dot{\ca O}^{\x}$; that is, $\mf{B}$ is the subgroup of $\mf{M}$ consisting of all exponentially bounded monomials whose multiplicative inverse is exponentially bounded.
Note that $L_{\alpha}(x) \in \mf{B}$ for all $\alpha \in \On$; in particular, $\mf{B}$ is a proper class.
Now let
\[
\T^*\ \coloneqq\ \{ f \in \HH : \supp f \subseteq \mf{B} \},
\]
where $\supp f$ denotes the support of the series $f \in \HH$.
Clearly, $\T^* \subseteq \dot{\ca O}$.
We also have $\LL \subseteq \dot{\ca O}$.

\begin{lem}
The structure $\T^*$ is a maximal subfield of~$\dot{\ca O}$.
\end{lem}
\begin{proof}
Since $\T^*$ consists of series in $\HH$ with support in $\mf{B}$ and $\mf{B}$ is a group, it follows from results of Neumann \cite{neumann} that $\T^*$ is a field. 

Moreover, $\T^*$ is a maximal subfield of $\dot{\ca O}$.
To see this, suppose that we have a subfield $L \subseteq \dot{\ca O}$ with $\T^*\subsetneq L$ and take $f \in L \setminus \T^*$.
Let $\mf{m}$ range over $\supp f$ and take $\mf{n} \in \supp f$ with $\mf{n}^{-1} \notin \dot{\ca O}$ such that $\mf{m} \in \mf{B}$ for all $\mf{m}$ with $\mf{m} \succ \mf{n}$.
Then $f = \sum_{\mf{m}\succ \mf{n}} f_{\mf{m}}\mf{m} + f_{\mf{n}}\mf{n} + \sum_{\mf{m} \prec \mf{n}} f_{\mf{m}}\mf{m}$ and $\sum_{\mf{m}\succ \mf{n}} f_{\mf{m}}\mf{m} \in \T^*$, so we can arrange that $f \sim f_{\mf{n}}\mf{n}$.
But this yields $f^{-1} \sim f_{\mf{n}}^{-1}\mf{n}^{-1} \notin \dot{\ca O}$, contradicting that $L$ is a subfield of~$\dot{\ca O}$.
\end{proof}

The proof of the next lemma was supplied by Vincent Bagayoko. It involves technical notions that we define only in the context needed here.
\begin{lem}\label{lem:T*diffsubfield}
The field $\T^*$ is a differential subfield of~$\dot{\ca O}$.
\end{lem}
\begin{proof}
Let $\LL$ be the differential field of logarithmic hyperseries and $\mf{L}$ its monomial group.
The (operator) support of $\der \colon \LL \to \LL$ is defined in \cite[Section~2.6]{loghyperseries} to be the smallest $\mf{S} \subseteq \mf{L}$ such that $\supp\der(\mf{l}) \subseteq \mf{S}\mf{l}$ for all $\mf{l} \in \mf{L}$ and shown to be $\{ \ell_{\gamma}^{\dagger} : \gamma \in \On \}$ in \cite[Lemma~3.1]{loghyperseries} and the subsequent remark.
This definition makes sense also for the extension of $\der$ to $\der \colon \HH \to \HH$, with $\mf{L}$ replaced by the monomial group $\mf{M}$ of $\HH$.
In \cite[Definition~1.14]{hyperseries-deriv}, the notion of the support of $\der$ is generalized to near-support, as we explain momentarily.
Then \cite[Proposition~6.14]{hyperseries-deriv} shows that $\{ \prod_{1 \les i \les n} \ell_{\gamma_i}^{\dagger} : \gamma_1, \dots, \gamma_{n} \in \On\}$ is a good near-support of $\der$ in $\LL$, where ``good'' means that it is a well-based subclass of $\ca O_{\LL}$ that is closed under finite products.
Goodness is used here only to apply \cite[Theorem~6.7]{hyperseries-deriv}, which gives that $\{ \prod_{1 \les i \les n} \ell_{\gamma_i}^{\dagger} : \gamma_1, \dots, \gamma_{n} \in \On\}$ remains a near-support of $\der \colon \HH \to \HH$.
This means that for $\mf{b} \in \mf{B}$ and $\mf{n} \in \supp \der(\mf{b})$, we have $\gamma_1, \dots, \gamma_n \in \On$ and $\mf{m} \in \mf{M}$ such that $\mf{m} \flatter \mf{b}$ and $\mf{n}=\mf{b}\mf{m}\prod_{1 \les i \les n}\ell_{\gamma_i}^{\dagger}$.
That $\mf{m} \flatter \mf{b}$ means $(\max\{ \mf{m},\mf{m}^{-1} \})^n < \max\{\mf{b},\mf{b}^{-1}\}$ for all $n$, so in particular $\mf{m} \in \mf{B}$.
Thus $\mf{n} \in \mf{B}$, so $\der(\mf{B})\subseteq \T^*$.
It follows that $\T^*$ is closed under~$\der$.
\end{proof}

\begin{prop}
The structure $(\HH, \T^*)$ is a transserial tame pair.
\end{prop}
\begin{proof}
Since $\T^*$ is a maximal differential subfield of $\dot{\ca O}$, it maps isomorphically as a differential field onto $\res(\HH,\dot{\ca O})$ under the residue map by Proposition~\ref{adh:7.1.3}.
In particular, $\T^* \models T^{\nl}_{\sm}$ and $(\HH, \T^*)$ is a transserial tame pair by Corollary~\ref{cor:tamepairequiv}.
\end{proof}

\begin{cor}
The structures $(\HH, \T^*, \ca O_{\T^*}, \dot{\ca O})$ and $(\HH, \ca O_{\HH}, \dot{\ca O})$ are model complete in their respective languages.
\end{cor}
\begin{proof}
These follow from Corollary~\ref{cor:modcompTTP} and Theorem~\ref{thm:modcompTnldhl} by the above. 
\end{proof}

Perhaps a similar argument can be carried out in the field of surreal numbers equipped with the derivation from \cite{bm-surreals}, into which the differential field $\T$ can be elementarily embedded by \cite{adh-surreals}.

\section*{Acknowledgements}
Thanks are due to Lou van den Dries for suggesting the topic of transserial tame pairs and for a helpful conversation.
Thanks are due to Vincent Bagayoko for a helpful conversation concerning Section~\ref{sec:hyperseries}, particularly providing the proof of Lemma~\ref{lem:T*diffsubfield}.
Thanks are due to the anonymous referee for carefully reading the paper and suggesting many improvements.

This material is based upon work supported by the National Science Foundation under Grant No.\ DMS-2154086.
This research was funded in whole or in part by the Austrian Science Fund (FWF) 10.55776/ESP450. For open access purposes, the author has applied a CC BY public copyright licence to any author accepted manuscript version arising from this submission.


\printbibliography

\end{document}